\documentclass[11pt,reqno]{amsart}

\usepackage{amssymb, amsmath, amsthm}
\usepackage{hyperref}
\usepackage[hyphenbreaks]{breakurl}
\usepackage{verbatim}
\usepackage{amscd} 
\usepackage[all]{xy}
\usepackage{tikz-cd}
\usepackage{cases}
\usepackage{array}
\usepackage{calligra,mathrsfs}

\makeatletter
\@namedef{subjclassname@2020}{%
  \textup{2020} Mathematics Subject Classification}
\makeatother

\title{Holonomic functions and prehomogeneous spaces}

\author{Andr\'as Cristian L\H{o}rincz}
\address{Humboldt--Universit\"at zu Berlin, Institut f\"ur Mathematik, Berlin, Germany}
\email[Andr\'as Cristian L\H{o}rincz]{lorincza@hu-berlin.de}


 \newtheorem{theorem}{Theorem}[section]
 \newtheorem{lemma}[theorem]{Lemma}
 \newtheorem{corollary}[theorem]{Corollary}
 \newtheorem{prop}[theorem]{Proposition}
 \newtheorem{definition}[theorem]{Definition}
 \newtheorem{example}{Example}[section]

\theoremstyle{remark}
 \newtheorem{remark}[theorem]{Remark}

\newcommand{\defi}[1]{{\upshape\sffamily #1}}
\DeclareMathOperator{\ShHom}{\mathscr{H}\text{\kern -3pt {\calligra\large om}}\,}

\newcommand{\C}{\mathbb{C}}

\newcommand{\bo}{\bigoplus}

\newcommand{\D}{\mathcal{D}}

\renewcommand{\det}{\textrm{det}}
\newcommand{\lie}{\mathfrak{g}}
\newcommand{\G}{\Gamma}

\renewcommand{\ll}{\lambda}
\newcommand{\LL}{\Lambda}

\newcommand{\oo}{\otimes}

\newcommand{\holo}{\mc{O}_X^{an}}

\newcommand{\Gf}{\mc{G}^{fin}}
\newcommand{\rank}{\operatorname{rank}}
\newcommand{\Sol}{\operatorname{Sol}}

\newcommand{\F}{\mathcal{F}}

\renewcommand{\O}{\mathcal{O}}

\newcommand*{\conn}{\mathscr{C}\text{\kern -3pt {\calligra\large onn}}}
\newcommand*{\alg}{\mathscr{A}\text{\kern -3pt {\calligra\large lg}}}

\newcommand{\Ann}{\operatorname{Ann}}

\newcommand{\Ext}{\operatorname{Ext}}

\newcommand{\GL}{\operatorname{GL}}

\newcommand{\Hom}{\operatorname{Hom}}
\newcommand{\End}{\operatorname{End}}

\newcommand{\SL}{\operatorname{SL}}

\newcommand{\Sing}{\operatorname{Sing}}

\newcommand{\Sym}{\operatorname{Sym}}

\newcommand\charC{{\operatorname{CharC}}}
\newcommand{\Char}{\operatorname{Char}}
\newcommand{\Supp}{\operatorname{Supp}}

\newcommand{\codim}{\operatorname{codim}}

\renewcommand{\det}{\operatorname{det}}

\renewcommand{\ker}{\operatorname{ker}}
\newcommand{\opmod}{\operatorname{mod}}
\newcommand{\opMod}{\operatorname{Mod}}
\newcommand{\op}{\operatorname}

\newcommand{\bb}[1]{\mathbb{#1}}

\newcommand{\mc}[1]{\mathcal{#1}}
\newcommand{\mf}[1]{\mathfrak{#1}}
\newcommand{\ol}[1]{\overline{#1}}

\newcommand{\ul}[1]{\underline{#1}}

\numberwithin{equation}{section}

\begin{document}

\begin{abstract} 
A function that is analytic on a domain of $\mathbb{C}^n$ is holonomic if it is the solution to a holonomic system of linear homogeneous differential equations with polynomial coefficients. We define and study the Bernstein-Sato polynomial of a holonomic function on a smooth algebraic variety. We analyze the structure of certain sheaves of holonomic functions, such as the algebraic functions along a hypersurface, determining their direct sum decompositions into indecomposables, that further respect decompositions of Bernstein-Sato polynomials. When the space is endowed with the action of a linear algebraic group $G$, we study the class of $G$-finite analytic functions, i.e. functions that under the action of the Lie algebra of $G$ generate a finite dimensional rational $G$-module. These are automatically algebraic functions on a variety with a dense orbit. When $G$ is reductive, we give several representation-theoretic techniques toward the determination of Bernstein-Sato polynomials of $G$-finite functions. We classify the $G$-finite functions on all but one of the irreducible reduced prehomogeneous vector spaces, and compute the Bernstein-Sato polynomials for distinguished $G$-finite functions. The results can be used to construct explicitly equivariant $\mathcal{D}$-modules.
\end{abstract}

\maketitle

\section*{Introduction}\label{sec:intro}

The solution-space of an overdetermined (i.e. holonomic) system of algebraic linear differential equations is finite dimensional. The solutions are special analytic functions, that are also called holonomic. Conversely, given a holonomic function, it is a basic problem to find a (especially the largest) corresponding holonomic system of algebraic linear differential equations that the function satisfies. This type of interplay between functions and $\D$-modules is the starting point of the paper.

While being analytic by definition, holonomic functions carry finite representability given solely through algebraic means, therefore suitable for various calculations (cf. \cite{zeilberg}, \cite{hgm}), and for algebro-geometric purposes in general. We call an analytic function defined on a domain of a smooth algebraic variety $X$ holonomic if it generates a holonomic $\D_X$-module. The basic property of such a holonomic $\D_X$-module is that it is $\mc{O}_X$-torsionfree (or Weyl-closed), which is rather well-behaved (cf. Corollary \ref{cor:singtor}), while the process of taking Weyl closure being essentially just a localization problem (cf. Theorem \ref{thm:codim2}).

Locally, holonomic functions can be always viewed as holonomic on the affine space through an \'etale coordinate transformation. Globally, they can be viewed as (multivalued) functions away from an algebraic hypersurface in $X$. We define the Bernstein--Sato polynomial of a holonomic function intrinsically (see Definition \ref{def:bho}), that depends on the holonomic function itself and the ambient smooth variety $X$. Its existence goes back to the work of M. Kashiwara. This definition is natural as it also generalizes the case of polynomials, and is inextricably linked to monodromy, being capable of detecting functions (and $\D$-modules) of geometric origin through rationality of roots (see Proposition \ref{prop:rat} and Corollary \ref{cor:integral}). As always, Bernstein--Sato polynomials are essential in understanding the $\D$-module structure of localizations (see Lemma \ref{lem:bcomp}) and the structure of $V$-filtrations \cite{vfilkash},\cite{vfilmal},\cite{sabbah3}.

Moreover, Bernstein--Sato polynomials are computable on an affine space by methods based on (non-commutative) Gr\"obner bases, e.g. implemented in the computer algebra systems \cite{M2}, \cite{singular}. Alternatively, they can also be computed through analytic expansions of holonomic functions. The latter point of view is essential for our calculations in Theorem \ref{thm:paramb1}, as the analogous computations based on algorithms for $\D$-modules failed to terminate in that case. Other analytic methods can be further used for finding roots of Bernstein--Sato polynomials of holonomic functions (e.g. Proposition \ref{prop:binroot}). 

In Section \ref{sec:holofun} we define various sheaves of holonomic functions that provide convenient ambient working objects. These are quasi-coherent $\mc{O}_X$-algebras equipped with a $\D_X$-module and monodromy action, and are intrinsic in nature (see Proposition \ref{prop:connsheaf}). An interesting case of holonomic functions are the algebraic functions, that is, functions that (locally) satisfy a polynomial relation with coefficients in $\mc{O}_X$. In particular, the sheaf of algebraic functions along a hypersurface is completely understood (see Theorem \ref{thm:algdecomp}) which, for instance, can be used for determining Bernstein--Sato polynomials (cf. Proposition \ref{prop:bfunalg}).

Starting from Section \ref{sec:gfin}, we consider functions in an equivariant setting. We introduce and study a class of functions called $G$-finite functions, that corresponds naturally to the class of equivariant $\D$-modules by Lemma \ref{lem:gfindmod}. In Section \ref{sec:bin0} we consider $G$-finite functions on binary forms that are built from roots of generic polynomials, which are of fundamental importance as all algebraic functions can be obtained from them via pullback. In Section \ref{sec:multfree}, we provide a multiplicity-free criterion that is our main technique for calculating Bernstein--Sato polynomials (of one or several variables) of $G$-finite functions. It is based on the original ideas of M. Sato for computing $b$-functions of semi-invariants for prehomogeneous vector spaces.

If the space $X$ has a dense orbit (i.e. it is prehomogeneous), $G$-finite functions turn out to be algebraic functions, and so the algebra of $G$-finite functions has a transparent structure (Corollary \ref{cor:gfinalg}). We introduce the notion of witness representations (see Section \ref{sec:red}), that can be thought of the right generalization of semi-invariants (or relative invariants) of prehomogeneous vector spaces: they satisfy the multiplicity-free condition, their Bernstein--Sato polynomials are symmetric with respect to duality, they detect the multiplicity of torsion-free simples in the composition series of torsion-free equivariant modules, they allow constructing simple torsion-free equivariant modules through explicit presentations etc. We find witness representations via a mix of representation theory and Bernstein--Sato polynomials, cf. Proposition \ref{prop:witcri}. 

In the last section, we provide a classification of $G$-finite functions, witness representations, and their Bernstein--Sato polynomials for all but one of the irreducible reduced prehomogeneous vector spaces (as in \cite{saki}). This is in analogy with the classification of $b$-functions of semi-invariants (see \cite{kimu}) which has been completed based on sophisticated microlocal techniques \cite{skko}. Further, we give a relation between Bernstein--Sato polynomials of $G$-finite functions related under castling (cf. Theorem \ref{thm:castle}), which can be used to extrapolate the results to infinitely many irreducible prehomogeneous vector spaces.

Historically, semi-invariants on prehomogeneous vector spaces were the first class of polynomials for which Bernstein--Sato polynomials were studied systematically. As far as we are aware, the $G$-finite functions introduced in this work constitute the first non-polynomial class of functions for which the same systematic study has been undertaken. As we further demonstrate, much of the theory of prehomogeneous vector spaces, that was initiated by M. Sato, can be extended naturally to $G$-finite functions. Such extensions should be possible in other aspects of the theory as well (see \cite{preh}). 

In the same spirit, we expect that many results on Bernstein--Sato polynomials have extensions from the case of polynomials to various classes of holonomic functions. 

Part of our motivation for these considerations comes from the problem of building equivariant $\D$-modules on representations with finitely many orbits in an explicit manner (see Open Problem 3 in \cite[Section 6]{mac-vil} and Section \ref{sec:bin0}), as this can lead to solving further complicated problems, e.g. determining local cohomology modules and related invariants (see \cite{bindmod, mike, senary, lHorincz2018iterated, locsub,perlyub}). One of the main ingredients in all of these articles is the Bernstein--Sato polynomial of semi-invariants, since it gives a filtration of localizations \cite[Proposition 4.9]{catdmod}, which is essential for understanding the category of equivariant $\D$-modules. The computations in Section \ref{sec:hard} will be used for this very same purpose in a subsequent work (see also Section \ref{sec:bin}).

\section{Holonomic functions and their Bernstein--Sato polynomials}\label{sec:prelim}

Let $X$ be a smooth, connected complex algebraic variety and $X^{an}$ the associated complex manifold. We write $\D_X$ for the sheaf of algebraic differential operators on $X$. We denote by $\opMod(\D_X)$ (resp. $\opmod(\D_X)$) the category of all $\O_X$-quasi-coherent (resp. $\D_X$-coherent) $\D_X$-modules. Throughout $\holo$ denotes the sheaf of holomorphic functions on $X^{an}$, and $\D_X^{an}$ the sheaf of differential operators on $X^{an}$ with holomorphic coefficients. Recall that there is a faithfully exact analytification functor that sends an $\O_X$-module (resp. $\D_X$-module) $\mc{M}$ to a $\holo$-module (resp. $\D_X^{an}$-module) $\mc{M}^{an}$. We say a subset $\Omega\subset X$ is a domain if it is a connected and open subset of $X^{an}$ in the complex analytic topology.

\subsection{Weyl closure}\label{sec:weylclosure}

For a coherent $\D$-module $\mc{M}$, we denote by $\Char \mc{M} \subset T^* X$ the characteristic variety of $\mc{M}$. The characteristic cycle $\charC(\mc{M})$ is the sum of the irreducible components of $\Char \mc{M}$ counted with multiplicities. Let $\pi: T^*X \to X$ denote the projection. We define the \defi{singular locus} of $\mc{M}$ to be 
\[
\Sing \mc{M}:=\ol{\pi(\Char \mc{M} \setminus T_{X}^*X)} \subset X.
\]
If the support $\Supp \mc{M}$ of $\mc{M}$ is not $X$, then $\Supp \mc{M} = \Sing \mc{M}$. Let $\xi$ be the generic point of $X$. The \defi{rank} of a $\D$-module $\mc{M}$ is defined to be
\[
\rank \mc{M} = \dim_{\O_{X,\xi}}  \mc{M}_\xi.
\]
In turn, $\rank \mc{M}$ is the multiplicity of the zero section $T_X^* X$ in $\charC \mc{M}$. Another interpretation of rank is in terms of solutions. For a domain $\Omega\subset X$ with embedding $i: \Omega \to X$,  let
\begin{equation}\label{eq:sol}
\Sol_\Omega(\mc{M}):=\Hom_{\D_X}(\mc{M}, i_*^{an}\mc{O}_{\Omega}^{an}) = \Hom_{\D^{an}_\Omega}(\mc{M}^{an}_{|\Omega}, \mc{O}_{\Omega}^{an})
\end{equation}
denote the space of solutions to $\mc{M}$ (here we view  $i_*^{an}\mc{O}_{\Omega}^{an}$ as a $\D_X$-module). When $X$ is affine, then $\Sol_\Omega(\mc{M})=\Hom_{\D_X}(\mc{M}, \holo(\Omega))$ and if $\mc{M}\cong D_{X}/\mc{I}$, the space $\Sol_\Omega(\mc{M})$ is identified with the space of analytic functions on $\Omega$ that are annihilated by $\mc{I}$. 

We begin with the following classical result (e.g. see  \cite[Section 4.3]{htt}).

\begin{theorem}[Cauchy--Kowalevskii--Kashiwara]\label{thm:ranksol}
Let $\mc{M}$ be a coherent $\D$-module and $\Omega\subset X\setminus \Sing \mc{M}$ be a simply-connected domain. We have
\[\dim_{\bb{C}} \Sol_\Omega(\mc{M}) = \rank \mc{M}.\]
\end{theorem}

We call a $\D_X$-module $\mc{M}$ a (algebraic, integrable) \defi{connection} on $X$, if it is locally free of finite rank as an $\O_X$-module. This is equivalent to $\Char \mc{M}=T^*_X X$ \cite[Proposition 2.2.5]{htt}.

Let $D\subset X$ be a union of (irreducible) hypersurfaces. Given a $\D$-module $\mc{M}$, we denote by 
\[\mc{M}(*D):= j_*j^*\mc{M}\]
the $\D_X$-module of the localization of $\mc{M}$ at $D$, where $j:X\!\setminus \! D\to X$ denotes the open embedding. We call $\mc{M}$ a (algebraic, integrable) \defi{meromorphic connection} along $D$, if $\mc{M}\cong j_* \mc{N}$, where $\mc{N}$ is a connection on $U$. As can be easily seen, a holonomic $\D_X$-module $\mc{M}$ is a meromorphic connection along $D$ if and only if $\Sing \mc{M}\subset D$ and $\mc{M}=\mc{M}(*D)$.

Given a $\D_X$-module $\mc{M}$, its $\O_X$-torsion subsheaf is the sheaf generated by all the torsion sections of $\mc{M}$. It is again a $\D_X$-module. Next, we introduce some terminology, following \cite{sst}.

\begin{definition}\label{def:weylclosure}
Let $\mc{M}$ be a $\D_X$-module. We define the \defi{Weyl closure} $\mc{M}^w$ of $\mc{M}$ to be the quotient of $\mc{M}$ by its torsion subsheaf. When $\mc{I}\subset \D_{X}$ is an ideal, the ideal $\mc{I}^w:= \D_X \cap (\O_{X,\xi} \cdot \mc{I})$ in $\D_{X}$ is called the \defi{Weyl closure} of $\mc{I}$.
\end{definition}

Note that when $\mc{M}=\D_{\C^n}/\mc{I}$, we have $\mc{M}^w = \D_{\C^n}/\mc{I}^w$ -- in fact, it is known that any holonomic $\D_{\C^n}$-module is cyclic. Further, an algorithm for computing Weyl closure in this case is implemented in \defi{Macaulay2} \cite{M2} and \defi{Singular} \cite{singular} (see \cite{leytsai}, \cite{sing2}, \cite{tsai}).

We have $\mc{M}=\mc{M}^w$ if and only if $\mc{M}$ is a torsion-free $\O_X$-module. The corresponding ideals are sometimes called Weyl-closed in the literature. Weyl closure satisfies the following universal property. If $f: \mc{N}\to \mc{M}$ is a map of $\D$-modules, and $\mc{M}$ is  $\O_X$-torsion-free, then $f$ factors through $\mc{N}^w$ as the composition $\mc{N} \to \mc{N}^w \to \mc{M}$.

Clearly, $\rank \mc{M}=\rank \mc{M}^w$, and $\mc{M}^w$ can be characterized as the smallest quotient of $M$ with this property (when $\rank \mc{M} < \infty$).

Among others, the importance of Weyl closure comes from finding annihilators of functions (e.g. see Lemma \ref{lem:witpres}). When $X$ is affine, given an ideal $\mc{I}\subset \D_X$ and $h\in \Sol_{\Omega}(\mc{I})$, clearly $\mc{I}^w\subset \Ann(h)$. It is natural to investigate the case when equality holds.

\begin{lemma}\label{lem:weylsol}
Let $X$ be affine and $\mc{M}=\D_X/\mc{I}$. If $\Omega\subset X$ is a domain such that $\dim \Sol_\Omega(\mc{I})= \rank M < \infty$, then $I^w \cong \Ann(\Sol_\Omega(\mc{I}))$. Moreover, $\Ann(h)=\mc{I}^w$ for any non-zero $h\in \Sol_\Omega(\mc{I})$ if and only if $(\bb{D}(\mc{M}^w))^w$ is a simple $\D_X$-module.
\end{lemma}

\begin{proof}
The first part is immediate (see \cite[Proposition 2.1.9]{tsai}).
Next, take any non-zero $h\in \Sol_\Omega(\mc{I})$.  We have an exact sequence of the form
\[ 0 \to \mc{K} \to \mc{M}^w \to \D_X/\Ann(h) \to 0.\]
Clearly, $\rank \mc{K} < \rank \mc{M}$, and $\mc{I}^w=\Ann(h)$ if and only if $K=0$. But $\mc{M}^w$ has no non-zero submodules with smaller rank if and only if the holonomic dual $\bb{D}(\mc{M}^w)$ has no non-zero submodules of smaller rank. The latter happens if and only if $(\bb{D}(\mc{M}^w))^w$ has no non-trivial submodules.
\end{proof}

The following is a consequence of the fact that if $\mc{M}$ has finite rank then the support of the torsion submodule of $\mc{M}$ is contained in $\Sing \mc{M}$.

\begin{lemma}\label{lem:weylclosure}
Let $\mc{M}$ be a coherent $\D$-module with finite rank and $Z$ any closed subset of $X$ with $\Sing \mc{M}\subseteq Z\subsetneq X$. Denote by $j:X\setminus Z \to X$ the open embedding and consider the exact sequence of $\D$-modules
\[0 \to \Gamma_Z (\mc{M}) \to \mc{M} \xrightarrow{\alpha} j_*j^*(\mc{M}) \xrightarrow{\beta} \mc{H}^1_Z(\mc{M}) \to 0.\]
Then $\mc{M}^w \cong \op{im} \alpha = \ker \beta =\mc{M}/\Gamma_Z(\mc{M})$.
\end{lemma}

A consequence of the above lemma is the following (see also \cite[Theorem 3.2.3]{htt}).
\begin{lemma}\label{lem:rankhol}
If $\mc{M}$ is a coherent $\D$-module, then $\rank \mc{M}$ is finite if and only if $\mc{M}^w$ is holonomic.
\end{lemma}

Monodromy can be used to decide the irreducibility of $\D$-modules as follows.

\begin{lemma}\label{lem:simpmon}
Let $\mc{M}$ be a holonomic $\D$-module with $\rank \mc{M} \in \bb{Z}_{>0}$, and denote by $\rho: \pi_1 (X\setminus \Sing \mc{M}) \to \GL_{\rank \mc{M}}(\C)$ the corresponding monodromy representation. Then the following statements hold:
\begin{itemize}
\item[(a)] Assume that the representation $\rho$ is irreducible. Then $\mc{M}$ is a simple $\D$-module if and only if both $\mc{M}$ and its dual $\bb{D} \mc{M}$ are $\O_X$-torsion-free.
\item[(b)] Conversely, assume that $\mc{M}$ is a regular holonomic simple $\D$-module. Then $\rho$ is irreducible.
\item[(c)] Assume that $\mc{M}$ is $\mc{O}_X$-torsion-free. If $\rho$ is indecomposable, then so is $\mc{M}$.
\end{itemize}
\end{lemma}

The case $\rank \mc{M}=1$ in part (a) gives a criterion to verify in a computable way the irreducibility of a $\D_{\bb{C}^n}$-module, e.g.  used in the proof of \cite[Proposition 3.6]{stat}. We have the following (compare with \cite[Proposition 4.8.16]{bjork}).

\begin{lemma}\label{lem:entire}
Let $\mc{M}$ be a coherent $\D_X$-module. Then $\mc{M}$ is a connection if and only if it is $\O_X$-torsion-free and $\dim \Sol_\Omega(\mc{M})=\rank \mc{M}<\infty$ for any simply connected domain $\Omega\subset X$.
\end{lemma}

\begin{proof}
If $\mc{M}$ is a connection, it is $\O_X$-torsion-free with $\Sing \mc{M}= \emptyset$ and all solutions can be extended to any simply-connected domain by Theorem \ref{thm:ranksol}.

Now for the other direction, let $\mc{M}$ be $\O_X$-torsion-free with $\dim \Sol_\Omega(\mc{M})=\rank \mc{M} <\infty$ for all simply-connected domains $\Omega$. Assume by contradiction that $\Sing \mc{M}$ is not empty, and take a simply-connected $\Omega$ with $\Sing \mc{M} \cap \Omega \neq \emptyset$. Since $\O_X^{an}$ is flat over $\O_X$, the module $\mc{M}^{an}_{|\Omega}$ is also $\O^{an}_\Omega$-torsion-free. Moreover, $\mc{M}^{an}_{|\Omega}$ is not an analytic connection as $\Char (\mc{M}^{an})= (\Char \mc{M})^{an}$, and by assumption it has a basis of $\rank \mc{M}$ solutions on $\Omega$. Such a basis yields a map $\mc{M}^{an}_{|\Omega} \to (\O^{an}_\Omega)^{\oplus \rank \mc{M}}$ with kernel $\mc{K}$ satisfying $\rank \mc{K} =0$. Since $\mc{M}^{an}_{|\Omega}$ is $\O^{an}_\Omega$-torsion-free, we obtain $\mc{K}=0$. But this implies that $\mc{M}^{an}_{|\Omega}$ is an analytic connection, a contradiction. Hence, $\mc{M}$ is an (algebraic) connection.
\end{proof}

The following is a basic structural result on the effects of Weyl closure.

\begin{theorem}\label{thm:codim2}
Let $\mc{M}$ be coherent $\D_X$-module of finite rank. Put $\Sing \mc{M} = D \, \bigcup \, C$, where $D$ is a hypersurface and $\codim_X C \geq 2$, and let $j:X\setminus \Sing \mc{M} \to X$ be the open embedding. Then $j_*j^*(\mc{M})$ is a meromorphic connection along $D$ and $\Sing \mc{M}^w \subset D$. Furthermore, if $\mc{M}(*D)$ is holonomic then have a decomposition of $\D_X$-modules
\[\mc{M}(*D) \,\, = \,\, \Gamma_C(\mc{M}(*D))\, \oplus \, j_*j^*(\mc{M}).\]
\end{theorem}

\begin{proof}
Put $U'=X\setminus \Sing \mc{M}$ and $U= X\setminus D$, and let $j_1: U' \to U$ and $j_2: U\to X$ be the open embeddings. We first show that the holonomic $\D_U$-module $\mc{N}:=j_{1*}j^*\mc{M}$ is a connection by using Lemma \ref{lem:entire}. Take any simply-connected domain $\Omega\subset U$. Since $\codim_X C\geq 2$, the open $\Omega\setminus (\Omega\cap C)$ is also simply-connected. By Theorem \ref{thm:ranksol} we have $\dim \Sol_{\Omega\setminus (\Omega\cap C)}(\mc{N}) = \rank \mc{N}$. Again, due to $\codim_X C\geq 2$, any holomorphic function on $\Omega\cap C$ can be extended to $\Omega$ by Hartogs' Theorem. Hence, by Lemma \ref{lem:entire} we obtain that $\mc{N}$ is a connection on $U$, and further that $j_*j^*(\mc{M})$ a meromorphic connection along $D$. Since $\mc{M}^w\subset j_*j^*(\mc{M})$, we obtain $\Sing \mc{M}^w \subset D$. 

Now assume that $\mc{M}(*D)$ is holonomic and consider the following exact sequence of holonomic $\D_{U}$-modules
\[0\to \Gamma_{C'}(j_2^*\mc{M}) \to j_2^*\mc{M} \to \mc{N} \to 0,\]
where $C'=C\setminus D$. Here the last map is surjective, as $\mc{N}\to \mc{H}^1_{C'}(j_2^*\mc{M})$ is the zero map, since $\mc{N}$ is a connection. It is enough to show that the sequence above splits, since from this the claim follows by applying $j_{2*}$. Using holonomic duality $\bb{D}$ it is enough to prove that for any simple holonomic $\D_U$-module $\mc{S}$ with $\codim_U\Supp \mc{S}\geq 2$ and any connection $\mc{Q}$ on $U$, we have $\Ext^1(\mc{S},\mc{Q})=0$. Assume that there is an exact sequence of the form
\begin{equation}\label{eq:ext10}
0\to \mc{Q} \to \mc{E} \to \mc{S} \to 0.
\end{equation}
Then $Z:=\Sing \mc{E} = \Supp \mc{S}$ is non-empty. Applying $\Gamma_{Z}$ to the sequence (\ref{eq:ext10}), we get an injective map $\Gamma_{Z} \mc{E} \to \mc{S}$. Assume by contradition that this map is not surjective. Then $\Gamma_{Z} \mc{E}=0$, which shows that $\mc{E}$ is $\O_U$-torsion-free. As we showed for $\mc{N}$ in the first half of the proof, this implies that $\mc{E}$ is a connection, a contradiction. Hence, $\mc{S}=\Gamma_{Z} \mc{E} \subset \mc{E}$ splits the sequence (\ref{eq:ext10}). 
\end{proof}

\begin{remark} The result above implies that Weyl closure $\mc{M}^w$ of a holonomic module $\mc{M}$ can be computed as the image of the composition $\mc{M}\to \mc{M}(*D) \to j_*j^* \mc{M}$, passing first through the localization map at $D$ and then the idempotent endomorphism corresponding to the projection onto the second factor of the decomposition in Theorem \ref{thm:codim2}.
\end{remark}

The following is a consequence of Lemma \ref{lem:entire} and Theorem \ref{thm:codim2}.

\begin{corollary}\label{cor:singtor}
A holonomic $\D$-module $\mc{M}$ is $\O_X$-torsion-free if and only if it is a submodule of a meromorphic connection. In this case $\Sing \mc{M}$ is a hypersurface, and for any domain $\Omega\subset X$ with $\Omega \cap \Sing \mc{M} \neq \emptyset$ we have $\dim \Sol_\Omega(\mc{M}) < \rank \mc{M}$.
\end{corollary}

Note that even if $h$ is an entire function on $X$, the locus $\Sing \D_X h$ can still be non-empty with $\D_{X} h$ being irregular holonomic (e.g. $h=\sin(x)/x$, see also \cite{stat}). 
We conclude the section with a useful fact on characteristic cycles \cite[Theorem 3.2]{ginz} (see also \cite[Lemma 5.6]{gyoja2}).

\begin{lemma}\label{lem:charc}
Let $\mc{M}$ be an $\O_X$-torsion-free regular holonomic module, and put $D=\Sing \mc{M}$. Then
\[\charC \left( \mc{M}(*D) \right) = \rank \mc{M} \, \cdot \, \charC \left(\O_X(*D)\right).\]
\end{lemma}

\subsection{Holonomic functions}\label{sec:holofun}

We continue with the notation from the previous section.

\begin{definition}\label{def:holonomic}
Let $\Omega\subset X$ be a domain and $h\in \holo(\Omega)$. We say that $h$ is holonomic if $D_X \cdot h$ is a holonomic $\D_X$-module.
\end{definition}

Clearly, $\D_X h$ is $\O_X$-torsion-free, and $\D_X h \cong \D_X/\Ann(h)$, where $\Ann(h)\subset \D_X$ denotes the ideal of differential operators that annihilate $h$. Conversely, when $X$ is affine then any holonomic $\O_X$-torsion-free $\D$-module $\mc{M}$ can be written as the quotient of $\D_X$ by the annihilating ideal of a $k$-tuple of functions, for some $k\leq \min\{ \rank \mc{M}, \, l(\mc{M})\}$ (see Lemma \ref{lem:weylsol}), where $l(\mc{M})$ denotes the $\D$-module length of $\mc{M}$. 

Holonomic functions satisfy pleasant (computable) closure properties \cite{zeilberg}, \cite{holotool}. For example, holonomic functions on a domain form an algebra that is also a $\D$-module. Pulling back holonomic functions along algebraic morphisms give holonomic functions. In fact, compositions of algebraic functions with holonomic functions are also holonomic. The geometric argument for this can be sketched as follows. Suppose we have a map $\bb{C}^m \to \bb{C}^n$ that is given by algebraic functions $a_1,\dots, a_n$, each satisfying a polynomial equation $p_1,\dots , p_n$ with coefficients in $\bb{C}[x_1,\dots,x_m]$. Let $Z$ be the (reduced) subvariety of $\bb{C}^{m+n}$ given by the vanishing of $p_i(x_1,\dots, x_m, y_i) \in \bb{C}[x_1,\dots,x_m, y_1,\dots,y_n]$, with $i=1,\dots,n$. Let $U$ be the smooth locus of $Z$ (away from the discriminantal locus). We have a diagram (with $p_1$ being \'etale)
\begin{equation}\label{eq:riemann}
\vcenter{\xymatrix@C-0.1pc@R-0.1pc{
& \ar[dl]_{p_1} \! U  \ar[dr]^{p_2} & \!\!\!\!\!\!\!\!\!\!\! \subset \, \bb{C}^{m+n} \\
\bb{C}^m & & \bb{C}^n
}}\end{equation}
If a $\D_{\bb{C}^n}$-module $\mc{M}$ has $h(y_1,\dots,y_n)$ as a solution on some domain, then $p_{1*} p_2^*(\mc{M})$ will have as a solution $h(a_1(x_1,\dots,x_m), a_2(x_1,\dots, x_m), \dots, a_n(x_1,\dots,x_m))$. Properties such as holonomicity, regularity, quasi-unipotent monodromy (see Proposition \ref{prop:rat}) follow by the corresponding functorial properties of $\D$-modules. More generally, the following holds.

\begin{lemma}\label{lem:etale}
Assume that a morphism of algebraic varieties $\phi: X_1 \to X_2$ is \'etale at $x\in X$, and $h$ is holonomic on a domain $\Omega_1 \subset X_1$ containing $x$. Take a local analytic inverse  $\psi$ of $\phi$ on a neighborhood $\Omega_2 \subset X_2$ of $\phi(x)$ with $\psi(\Omega_2) \subset \Omega_1$. Then $h'=h \circ \psi$ is a holonomic function on $\Omega_2\subset X_2$.
\end{lemma}

\begin{proof}
We can assume that $\phi$ is an \'etale morphism of affine varieties. In this case, we have an injection $\phi^*:\D_{X_2} \to \D_{X_1}$ and the $\D$-module pushforward functor along $\phi$ is just given by restriction of scalars. Therefore, $\D_{X_1} \cdot h$ is holonomic as a $\D_{X_2}$-module. Hence, so is $\D_{X_2} \cdot h' \, \cong \, \phi^*(\D_{X_2}) \cdot h \, \subset \, \D_{X_1} \cdot h$.
\end{proof}

Thus, locally every holonomic function on a domain of $X$ can be written as a holonomic function on a domain of $\bb{C}^{n}$, with $n=\dim X$. Namely, by the Noether normalization lemma we can take an affine open $U\subset X$ and a finite \'etale morphism $\phi: U \to V \subset \bb{C}^{n}$, for some affine open on $V$. By the above, $\D_{\bb{C}^n} \cdot h'$ is then holonomic. The \defi{Mathematica} package \defi{HolonomicFunctions} is designed for computing with holonomic functions on an affine space \cite{Koutschan09}, \cite{holopack}.

\begin{definition}\label{def:bho}
Let $h$ be a holonomic function defined on some domain in $X$, and $D= \Sing (\D_X h)$ the associated reduced divisor  (see Corollary \ref{cor:singtor}). Consider $f \in \C[U]$ defining $D$ in some affine neighborhood $U \subset X$. Then the monic polynomial $b_{h,U}(s) \in \bb{C}[s]$ of minimal degree satisfying an equation of the form
\begin{equation}\label{eq:bu}
P(s) \cdot f^{s+1}h = b_{h,U}(s) \cdot f^s h,
\end{equation}
for some $P(s)\in \D_{U}[s]$, is called the Bernstein--Sato polynomial (or $b$-function) of the holonomic function $h$ on $U$. The (global) Bernstein--Sato polynomial $b_h(s)$ of $h$ on $X$ is 
\[b_h(s) := \underset{i \in I}{\op{lcm}} \,\, b_{h,U_i}(s),\]
for an open affine cover $\{U_i\}_{i\in I}$ of $X$ that trivializes $\O_X(D)$.
\end{definition}

Similarly, for a point $x\in X$ one can define the local Bernstein--Sato polynomial $b_{h,x}(s)$ of the holonomic function $h$ as the polynomial $b_{h,U}(s)$ of minimal degree in (\ref{eq:bu}) over affine open neighborhoods $x\in U \subset X$.

For the existence of $b_{h,U}(s)$ as well as the fact that $b_h(s)$ is well-defined, see \cite[Theorem 2.7]{kashiII}, \cite[Section 5]{gyoja2}, and \cite[Lemma 2.5.2]{gyoja}.

\begin{example}\label{ex:bfun}
The roots of the Bernstein--Sato polynomial of a holonomic function are not necessarily rational. When $X=\C$ and $h = \sin ( \log x)$, we have $b_h(s) = (s+1-i)(s+1+i)$ (see \cite[Section 3.3]{uli}).
\end{example}

With the knowledge of the annihilating ideal of $h$, the routine \defi{globalB} in \defi{Macaulay2} \cite{M2}, can be used to find $b_h(s)$ on an affine space (cf. \cite[Section 3.3]{uli}).

Rationality of the roots is guaranteed if the module is of geometric origin, due to the quasi-unipotent property. We say that a connection on an open $U\subset X$ has \defi{quasi-unipotent local monodromy} with respect to $X$, if for some embedded resolution of $(X, X\setminus U)$ the eigenvalues of the local monodromy operators around the divisors at infinity are roots of unity. The following is a basic result that is implicitly used in the theory of (rational) $V$-filtrations and vanishing cycles \cite{vfilkash}, \cite{vfilmal}, \cite{sabbah3}.

\begin{prop}\label{prop:rat}
Let $h$ be a holonomic function and put $U=X\setminus \Sing(\D_X h)$. Assume that the connection $\D_U h$ has quasi-unipotent local monodromy with respect to $X$. Then the roots of $b_h(s)$ are rational.
\end{prop}

We point out the slight discrepancy in terminology in the case of polynomials $p$, when the Bernstein--Sato polynomial $b_{p^{-1}}(s)$ of the holonomic function $p^{-1}$ is the Bernstein--Sato polynomial of $p$ in the usual sense (up to a shift in $s$). In this case the connection is the structure sheaf $\O_U$ with the rationality of roots established in \cite{kashi2}.
 
The following is a consequence of Proposition \ref{prop:rat}., as an algebraic function has finite monodromy, and integration preserves regularity and quasi-unipotence (e.g. see \cite{nilsson} for a more precise statement based on functions).

\begin{corollary}\label{cor:integral}
Let $h$ be a holonomic function such that $b_h(s)$ has an irrational root. Then $h$ cannot be represented as an integral over an algebraic function.
\end{corollary}

As an example, we determine the Bernstein--Sato polynomial of the hypergeometric function.

\begin{example}\label{ex:hyper}
Put $X=\bb{P}^1$ and consider the hypergeometric function $F(z) = \, _2F_1(a,b;c;z)$ that is annihilated (on $\bb{C}^1$) by the operator
\begin{equation}\label{eq:hyp}
z(z-1)\frac{d^2}{dz^2} + [(a+b+1)z-c] \frac{d}{dz} + ab.
\end{equation}
The regular holonomic $\D$-module $\D_X \cdot F$ has singular locus $\{0,1,\infty\}$. An elementary calculation (using either the operator (\ref{eq:hyp} or the hypergeometric series) shows that we have the equation
\[ Q\cdot z^{s+1}F = (s+1)(s+2-c) \cdot z^s F, \mbox{ where } \]
\[ Q=z(z-1)\frac{d^2}{dz^2} + [(a+b-2s-1)z-c+2s+2] \frac{d}{dz} + s^2+(2-a-b)s+(a-1)(b-1).\]
Thus, the local Bernstein--Sato polynomial of $F$ at $z=0$ is $b_{F,0} (s) = (s+1)(s+2-c)$. Around $z=1$ (resp. $z=\infty$), the function $\,_2F_1(a,b;1+a+b-c;1-z)$ (resp. $z^{-a}\,_2F_1(a,1+a-c;1+a-b;z^{-1})$) is annihilated by (\ref{eq:hyp}). From the result at $z=0$, we get $b_{F,1}(s)=(s+1)(s+1+c-a-b)$ and $b_{F,\infty}(s)=(s+a+1)(s+b+1)$. Hence, the global Bernstein--Sato polynomial of $F$ on $X$ is (assuming that the roots are distinct)
\[b_F(s) = (s+1)(s+1+a)(s+1+b)(s+2-c)(s+1+c-a-b).\]
In particular, the roots of $b_F(s)$ are rational if and only if $a,b,c \in \bb{Q}$. Thus, using Corollary \ref{cor:integral} and Euler's integral representation of $F$, we see that $F$ can be written as an integral over an algebraic function if and only if $a,b,c\in \bb{Q}$ (so that the equation given by (\ref{eq:hyp}) is of geometric origin). This method offers an alternative approach that avoids calculating monodromy operators explicitly.
\end{example}

Let $D\subset X$ be a hypersurface and $\mc{M}$ a $\D_X$-module such that $M$ is holonomic on $X\setminus D$. As in Definition \ref{def:bho} one can define the Bernstein--Sato polynomial $b_{m,D}(s)$ for any element $m \in \Gamma(X,M)$ \cite[Theorem 2.7]{kashiII}. If $\mc{M}$ is torsion-free and holonomic, we put $b_{m}(s):=b_{m,\, \Sing \D_X m}(s)$ -- equivalently, one can define the Bernstein--Sato polynomial of a tuple of holonomic functions in this way.

The following gives a criterion for normalization of Bernstein--Sato polynomials, analogous to the fact that $-1$ is a root of the Bernstein--Sato polynomial of a polynomial (the case $h=1$).

\begin{lemma}\label{lem:normalize}
Let $h$ be a holonomic function on a domain $\Omega \subset X$, and $D\subset X$ a hypersurface such that there exist $x \in \Omega \cap D$ with $h(x) \neq 0$. Then $b_{h,D}(-1)=0$.
\end{lemma}

\begin{proof}
We can assume that $X$ is affine and $D$ is defined by $f$.  We have an equation of the form $P \cdot f^{s+1}h = b_{h,D}(s) \cdot f^s h$. Putting $s=-1$ we get $P \cdot h = b_{h,D}(-1) \cdot h/f$. We plug $x$ into the equation. As $P \cdot h$ is an analytic function at $x$, and $h/f$ has a pole, we must have $b_{h,D}(-1) = 0$.
\end{proof}

In practice, it often happens that for a holonomic function $h$ it is possible to find some annihilating differential operators that generate an ideal of finite rank, but finding a complete set of generators for the entire annihilating ideal is difficult. In this case we have the following comparison result.

\begin{lemma}\label{lem:bcomp}
Let $h$ be a holonomic function with $D=\Sing D_X h$. Let $\mc{M}=\D_X/\mc{I}$ such that $\mc{I} \cdot h = 0$ and $M$ is holonomic on $X\setminus D$. Then $b_h(s)$ divides $b_{m,D}(s)$, where $m=\overline{1}\in \D_X/\mc{I}$.
\end{lemma}

\begin{proof}
The statement being local, we can assume that $X$ is affine and $D$ is defined by $f \in \bb{C}[X]$. Then for some $P\in \D_X[s]$, we have an equation $P\cdot f^{s+1}m =b_{m,D}(s) f^{s} m$. This clearly descends to an equation $P\cdot f^{s+1}h =b_{m,D}(s) f^{s} h$ via the induced map $\D_X[s] f^s m \to \D_X[s] f^s h$, hence $b_h(s) \, | \, b_{m,D}(s)$.
\end{proof}

The next result is an important application of Bernstein--Sato polynomials towards understanding the $\D$-module structure of localizations.

\begin{lemma}\label{lem:bgen}
Let $X$ be affine, and assume that $\mc{M}$ is a torsion-free holonomic $\D_X$-module and that $f \in \C[X]$ defines $\Sing \mc{M}$. For an element $m\in \mc{M}_f$ with $\mc{M}\subset \D_X m$, let $\alpha$ (resp $\beta$) denote the largest (resp. smallest) integer root of $b_m(s)$ (if none, put $\alpha=-1$ and $\beta=0$). Then $f^{\alpha+1} \cdot m \in \mc{M}$, and $f^\beta \cdot m$ generates $\mc{M}_f$ as a $\D_X$-module, with $\beta$ being the largest integer with this property (assuming $b_m(s)$ has at least one integer root).
\end{lemma} 

\begin{proof}
We have $\Sing (\D_X \cdot m) = \{f=0\}$. Hence, for some $P \in \D_X[s]$ we have an equation
\[ P \cdot f^{s+1}m = b_m(s) \cdot f^s m.\]
Clearly $f^k \cdot m \in \mc{M}$, for some integer $k\gg 0$. When $k>\alpha+1$, we put $s=k-1$ in this equation to get $f^{k-1} \cdot m \in \mc{M}$. Thus, $f^{\alpha+1} \in \mc{M}$. The part regarding $\beta$ follows from \cite[Proposition 4.2]{torrelli}.
\end{proof}

We denote by $F_\bullet \D_X$ the order filtration on $\D_X$. The following shows, in particular, that modulo the integers the roots are intrinsic to the $\D$-module.

\begin{lemma}\label{lem:brootgen}
Let $h_1,h_2$ be holonomic functions with an isomorphism $\phi: \D_X h_1 \xrightarrow{\cong} \D_X h_2$. Write $\phi(h_1) \in F_{d_2} \D_X \cdot h_2$ and $\phi^{-1}(h_2)  \in \, F_{d_1} \D_X \cdot h_1$, for some $d_1,d_2 \in \bb{N}$. Then $b_{h_1}(s)$ divides
\[b_{h_2}(s-d_2) \cdot b_{h_2}(s-d_2+1) \cdots b_{h_2}(s+d_1-1) \cdot b_{h_2}(s+d_1).\]
\end{lemma}

\begin{proof}
The statement being local, we can assume that $X$ is affine such that $f \in \C[X]$ defines $\Sing \mc{M}$. It is easy to see that for any operator $Q\in \D_X$, we have $\D_X[s] \cdot f^{s+\deg Q} (Q\cdot h_i) \subset \D_X[s] \cdot f^s h_i$. The argument follows now as in \cite[Lemma 5.13]{gyoja2}.
\end{proof}

\begin{lemma}\label{lem:bsum}
Let $D \subset X$ be a hypersurface and $h_1, h_2$ two holonomic functions on a domain. Assume that we have $\D_X h_1 \, \bigcap \, \D_X h_2 =0$. Then $\op{lcm} (b_{h_1,D}(s), b_{h_2,D}(s))$ divides $b_{h_1+h_2,D}(s)$.
\end{lemma}

\begin{proof}
We can assume that $X$ is affine and $D$ is defined by $f \in \bb{C}[X]$. Let $P \cdot f^{s+1}h =b_{h,D}(s) \cdot f^s h$, for $P\in \D_X[s]$. Note that the latter gives $A(s):=P \cdot f^{s+1} h_1 - b_{h,D}(s)f^s h_1 = - (P \cdot f^{s+1} h_2 - b_{h,D}(s)f^s h_2)$. Assume that $A(s) \neq 0$. Then we can evaluate $A(s)$  at some large $k\in \bb{N}$ such that $0 \neq A(k) \in \D_X h_1 \cap \D_X h_2$, a contradiction. Hence, we have $A(s)=0$, which implies that $b_{h_i,D}(s) \, \vert \, b_{h,D}(s)$, thus finishing the proof.
\end{proof}

Under some additional hypotheses, we have the following converse statement to the one above.

\begin{lemma}\label{lem:bsum2}
Let $D\subset X$ be a hypersurface, write $U=X\setminus D$, and consider holonomic functions $h_1,\dots, h_n$ on a domain. Assume that for any $1\leq i < j \leq n$, the $\D_U$-modules $\D_{U} h_i$ and $\D_{U} h_j$ have no common simple composition factors. Then for each $1\leq i \leq n$, there exist $d_i, p_i \in \bb{N}$ such that $h_i \in \mc{O}(k_i D) \cdot F_{d_i} \D_X \cdot h$, where $h=h_1 + \dots + h_n$. In such case, $b_{h , D}(s)$ divides $\underset{1\leq i \leq n}{\op{lcm}} b_{h_i,D}(s)\cdot b_{h_i,D}(s+1) \cdots b_{h_i,D}(s+d_i+p_i)$.
\end{lemma}

\begin{proof}
We can assume that $X$ is affine and $D$ is defined by $f \in \bb{C}[X]$. Consider the embedding $\D_U h \to \D_U h_1 \oplus \dots \oplus \D_U h_n$. Since the induced maps $\D_U h \to \D_U h_i$ via projections are surjective, $\D_U h $ contains all the composition factors of $\D_U h_i$, for $i=1,\dots,n$. The latter factors are disjoint by assumption, which implies that $\D_U h = \D_U h_1 \oplus \dots \oplus \D_U h_n$. Hence, we have an exact sequence
\[0 \to \D_X h \to \D_X h_1 \oplus \dots \oplus \D_X h_n \to L \to 0,\]
with $\op{supp} L \subset D$. By the Nullstellensatz, there exists some $p_i\in \bb{N}$ and $P_i \in \D_X$ such that $P_i \cdot h =  f^{p_i} \cdot h_i$, for all $i=1,\dots, n$. As seen in the proof of Lemma \ref{lem:brootgen} this implies that there is $Q_i \in \D_X[s]$ such that $Q_i \cdot f^{s+1} h = f^{s+d_i+p_i+1} \cdot h_i$, where $d_i$ is the degree of the operator $P_i$. 

Let $T_i(s) \cdot f^{s+1} h_i =b_{h_i,D}(s) \cdot f^s h_i$, for some $T_i(s) \in \D_X[s]$, where $1\leq i \leq n$. Then we have $Q \cdot f^{s+1}h = b(s) \cdot f^s h$, where $b(s)=\underset{1\leq i \leq n}{\op{lcm}} b_{h_i,D}(s)\cdot b_{h_i,D}(s+1) \cdots b_{h_i,D}(s+d_i+p_i)$ and
\[Q= b(s) \cdot \sum_{i=1}^n \left(\prod_{j=0}^{d_i+p_i} \frac{1}{b_{h_i,D}(s+j)} \cdot T_i(s+j) \right) \cdot Q_i.\]
\end{proof}

We recall that the solutions of regular holonomic $\D$-modules are of Nilsson class (i.e. of moderate growth, see \cite[Sections 4.8 and 4.9]{bjork},  \cite[Section 2.5]{sst}).

We now define sheaves of \lq\lq holonomic functions of Nilsson class on $X$". Fix a simply-connected domain $\Omega\subset X$ with open embedding $i: \Omega \to X$, and let $N\triangleleft \pi_1(X)$ be a normal subgroup. We define $\conn_X^N$ to be the sum of all regular (algebraic) connections in $i_*^{an} \mc{O}^{an}_\Omega$ with trivial $N$-monodromy. In more explicit terms, for a Zariski-open subset $U\subset X$, we have
\[
\conn_X^N(U)=\{h: U\cap \Omega \to \bb{C} \mid \D_U h \mbox{ regular connection with monodromy factoring via } \pi_1(U) \to \pi_1(X)/N \}.\]

More generally, for $D\subset X$ a hypersurface with open embedding $j: X\setminus D \to X$, let $\Omega\subset X\setminus D$ be a simply-connected domain, and consider a normal subgroup $N\triangleleft \pi_1(X\!\setminus\!D)$. We put $\conn_X^N(D):= j_* \conn_{X\!\setminus\!D}^N$, which stands for the sheaf of holonomic functions of Nilsson class with singularities along $D$ and trivial $N$-monodromy.

For two hypersurfaces $D_1 \subset D_2$, and $N_i\triangleleft \pi_1(X\!\setminus\!D_i)$ such that the image of $N_2$ under the map $\pi_1(X\!\setminus\!D_2) \to \pi_1(X\!\setminus\!D_1)$ is contained in $N_1$, we have $\conn_X^{N_1}(D_1) \subset \conn_X^{N_2}(D_2)$. For a hypersurface $D$, the minimal object is $\conn_X^{\pi_1(G)}(D) = \mc{O}_X(*D)$, while the maximal being $\conn_{X}(D):=\conn_{X}^{\{1\}}(D)$.

Let $\opMod_{D,w}^{N.rh}(\D_X)$ be the full subcategory of $\opMod(\D_X)$ consisting of modules $\mc{M}$ with the property that for any coherent $\D_X$-submodule $\mc{N}\subset \mc{M}$, its Weyl closure $\mc{N}^w$ is regular holonomic with $\Sing \mc{N}^w \subset D$ and trivial $N$-monodromy on $X\setminus D$. One can see that $\opMod_{D,w}^{N.rh}(\D_X)$ is closed under subquotients in $\opMod(\D_X)$.

\begin{prop}\label{prop:connsheaf}
$\conn_{X}^N(D)$ is an $\mc{O}_X$-quasi-coherent sheaf of $\D_X \times \pi_1(X\!\setminus\!D)/N$-algebras. Furthermore: 
\begin{itemize}
\item[(a)] $\conn_X^N(D)$ is an injective object in $\opMod_{D,w}^{N,rh}(\D_X)$.
\item[(b)] For all $\mc{M}$ in $\opMod_{D,w}^{N,rh}(\D_X)$, we have a natural isomorphism $\Sol_\Omega(\mc{M}) \cong \Hom_{\D_X}(\mc{M},\conn_X^N(D)).$ 
\item[(c)] $\conn_X^N(D)$ does not depend on the choice of $\Omega\subset X\setminus D$ (up to isomorphism).
\end{itemize}
\end{prop}

\begin{proof}
It is clear from its definition that $\conn_{X}^N(D)$ belongs to $\opMod_{D,w}^{N,rh}(\D_X)$. Take $\mc{M}\in\opMod_{D,w}^{N,rh}(\D_X)$. From (\ref{eq:sol}) we have an injective map $\phi: \Hom_{\D_X}(\mc{M},\conn_X^N(D))  \to \Sol_\Omega(\mc{M})$. Writing $\mc{M} = \bigcup \mc{M}_i$ with $\mc{M}_i$ coherent, we have
\[\Sol_\Omega(\mc{M}_i) \cong \Sol_\Omega(\mc{M}_i^w)=\Hom_{\D_X}(\mc{M}_i,\conn_X^N(D)),\]
for all $i$. Taking the limit with respect to $i$, we obtain that $\phi$ is an isomorphism, thus proving (b).

For part (a), take a monomorphism $\psi: \mc{N} \to \mc{M}$ in  $\opMod_{D,w}^{N,rh}(\D_X)$, and define an exhaustive filtration of $\mc{N}$ as $\mc{N}_i = \psi^{-1}(\mc{M}_i)$. Then $\mc{N}_i$ is also coherent, and we have monomorphisms $\psi_i : \mc{N}_i \to \mc{M}_i$. As $\rank$ is additive on holonomic modules, by part (b) we have surjective maps $\Hom_{\D_X}(\mc{M}_i,\conn_X^N(D)) \to \Hom_{\D_X}(\mc{N}_i,\conn_X^N(D))$ of finite dimensional spaces induced by $\psi$. Since the Mittag--Leffler conditions are satisfied, taking the limit with respect to $i$ gives a surjective morphism $\Hom_{\D_X}(\mc{M},\conn_X^N(D)) \to \Hom_{\D_X}(\mc{N},\conn_X^N(D))$ induced by $\psi$, showing (a).

For part (c) it is enough to take the case $D=\emptyset$. Let $\Omega' \subset X$ be another simply-connected domain. Denote by $\mc{C}$ (resp, $\mc{C}'$) the $\D_X$-module $\conn_X^N$ constructed on the domain $\Omega$ (resp. $\Omega'$). We will construct an isomorphism of $\D_X \times \pi_1(X)/N$-algebras between $\mc{C}$ and $\mc{C}'$ explicitly.

Fix the point $x$ (resp. $x'$) in $\Omega$ (resp. $\Omega'$), and take a smooth curve $\gamma:[0,1] \to X$ with $\gamma(0) = x$ and $\gamma(1)=x'$. Take an arbitrary Zariski-open subset $U \subset X$ with complement $Z=X\setminus U$. Consider a tubular neighborhood (i.e. normal tube) $T$ of the curve $\gamma$. Since $Z$ has real codimension $\geq 2$, $T\setminus S$ is still connected. Take $y \in \Omega \cap (T\setminus S)$ (resp. $y' \in \Omega \cap (T\setminus S)$), and let $\phi:[0,1] \to T\setminus S$ be a path with $\phi(0)=y$ and $\phi(1)=y'$. We define a morphism $\alpha_U$ from $\mc{C}(U)$ to $\mc{C}'(U)$ by analytic continuation of a section on $U\cap \Omega$ along the path $\phi$. Clearly, $\alpha_U$ is a well-defined isomorphism of $\Gamma(U,\D_X)\times \pi_1(X)/N$-algebras. Since the monodromy of a section of $\mc{C}(U)$ factors through $\pi_1(X)$ and $T$ is simply-connected, it follows that $\alpha_U$ does not depend on the choice of $y,y',\phi,T$. Due to this fact, it is now straightforward to see the compatibility of the construction with restriction maps, yielding an isomorphism $\alpha : \mc{C} \to \mc{C}'$ of sheaves.
\end{proof}

For a group $H$, we denote by $\opmod(H)$ the category of its finite dimensional complex representations. An object $M$ in $\opmod(H)$ is called a \defi{generator} if any $Z \in \opmod(H)$ has a surjection $M^{\oplus p} \to Z$ for some $p\in \bb{N}$.

\begin{theorem}\label{thm:generator}
The $\D_X$-module $\conn_{X}^{N}(D)$ is coherent if and only if $\opmod(\pi_1(X\!\setminus\!D)/N)$ has a generator. In this case, $\opmod(\pi_1(X\!\setminus\!D)/N)$ has finitely many irreducibles $\chi_1,\dots, \chi_k$ (up to isomorphism), whose projective covers correspond to regular meromorphic connections $\mc{I}_1,\dots ,\mc{I}_k$ on $X$ via the Riemann--Hilbert correspondence, and we have a decomposition into indecomposable $\D_X$-modules
\[\conn_X^N(D) \, \cong \, \bigoplus_{i=1}^k \, \mc{I}_i^{\, \oplus \dim \chi_i}.\]
\end{theorem}

\begin{proof}
Consider the case $D=\emptyset$. Assume first that $\conn_X^N$ is coherent, in which case it must be a regular connection, since it belongs to $\opMod_{D,w}^{N.rh}(\D_X)$. Denote by $P$ the representation of $\pi_1(X)/N$ which corresponds to $\conn_X^N$ by the Riemann--Hilbert correspondence. By Proposition \ref{prop:connsheaf}(a) $P$ is a projective module in $\opmod(\pi_1(X)/N)$. Moreover, by Proposition \ref{prop:connsheaf}(b)  $P$ surjects onto any irreducible representation of $\pi_1(X)/N$. This implies that $P$ is a (pro)generator in $\opmod(\pi_1(X)/N)$.

Conversely, assume that $\opmod(\pi_1(X)/N)$ has a generator. It follows by the Gabriel--Popescu theorem \cite{gabpop} that the category $\opmod(\pi_1(X)/N)$ is equivalent to the category of finite dimensional reprentations of a finite dimensional algebra (see also \cite[Theorem 2.11]{paquette}). In particular, $\opmod(\pi_1(X)/N)$ has finitely many simples $\chi_1,\dots, \chi_k$ corresponding to connections $\mc{S}_1, \dots, \mc{S}_k$, with projective covers $P_1,\dots,P_k$ (resp. injective envelopes $I_1,\dots, I_k$) in $\opmod(\pi_1(X)/N)$ and corresponding regular connections $\mc{I}_1,\dots ,\mc{I}_k$ (resp. $\mc{P}_1,\dots,\mc{P}_k$) on $X$.

Take an arbitrary coherent $\D_{X}$-submodule $\mc{M}\subset \conn_X^N$, which must be regular holonomic. We want to show that the length $l(\mc{M})$ of $\mc{M}$ is bounded above by $b:=\sum_{i=1}^k \rank \mc{P}_i$. For any $1\leq i \leq k$, we denote by $[\mc{M}:\mc{S}_i]$ the multiplicity of $\mc{S}_i$ in a composition series of $\mc{M}$. By Proposition \ref{prop:connsheaf}(c), we have the following 
\begin{equation}\label{eq:multrank}
[\mc{M}:\mc{S}_i] = \dim \Hom_{\D_X}(\mc{P}_i, \mc{M}) \, \leq \, \dim \Hom_{\D_X}(\mc{P}_i, \conn_X^N) = \rank \mc{P}_i,
\end{equation}
So $l(\mc{M})\leq b$, as we claimed. Since $\mc{M}$ was an arbitrary finitely generated $\D$-submodule of $ \conn_X^N$, this shows that $\conn_X^N$ is itself a regular connection, of length $b$. 

By Proposition \ref{prop:connsheaf}(b), we must have a decomposition
\[
\conn_X^N \, = \, \bigoplus_{i=1}^k \,\, \mc{I}_i^{\,\oplus  m_i},
\]
for some $m_i \in \bb{N}$. Using this, we obtain as in (\ref{eq:multrank}) the equalities
\[m_i = \dim \Hom_{\D_X}(\mc{S}_i, \conn_{N,X}) = \rank \mc{S}_i = \dim \chi_i.\]

When $D$ is not empty, we are left to show is that the pushforward $j_* \mc{I}$ along $j: X\!\setminus\!D \to X$ of an indecomposable holonomic $\D_{X\!\setminus D}$-module $\mc{I}$ is indecomposable. This follows by Lemma \ref{lem:simpmon} (c).
\end{proof}

\begin{remark}
As seen in the proof above, the condition that $\opmod(H)$ has a generator means that it is equivalent to the category of finite dimensional modules of a finite dimensional algebra. Since the requirement is only on the level of finite dimensional modules, this is a much weaker condition than Morita-equivalence -- in fact, the latter holds only for finite groups \cite[Theorem 1]{connell}. For example, the Higman group \cite{higman} has no non-trivial finite dimensional representations.
\end{remark}

Let $h$ be analytic in some domain $\Omega\subset X$. We say that the function $h$ is \emph{algebraic} on $X$ if for some (hence, any) open affine $U\subset X$, its restriction $\left.h\right|_{\Omega\cap U}$ satisfies a polynomial equation with coefficients in $\bb{C}[U]$. We define the \defi{sheaf of algebraic functions} $\alg_X$ on $X$ as we did for $\conn_{X}$ by requiring additionally $h$ to be algebraic in its definition.

A basic result in differential Galois theory is that a function $h$ is algebraic if and only if it is of Nilsson class with finite monodromy, in which case this coincides with its Galois group $\op{Gal}(K(h)/K)$, where $K=\bb{C}(X)$. 
In other words, we have as $\D_X \times \pi_1(X)$-modules
\begin{equation}\label{eq:algfin}
\alg_X = \displaystyle\sum_{\substack{N \triangleleft \, \pi_1(X) \\ |\pi_1(X):\,H| \text{ finite}}} \!\!\!\! \conn_X^N \qquad \subset \conn_X.
\end{equation}

Similarly, for a hypersurface $D$ with embedding $j:X\!\setminus\!D \to X$, we put $\alg_X(D)=j_* \alg_{X\!\setminus\!D}$ for the sheaf of algebraic functions with singularities along $D$. Note that the discriminant of a (monic) polynomial equation that is satisfied by a (local) section of $\alg_X(D)$ vanishes along $D$. 

We have a precise $\D$-module-theoretic description of $\alg_X(D)$ as follows. For a finite dimensional (complex) representation $\chi$ of $\pi_1(X\!\setminus\!D)$, let $\mc{S}^\chi$ be the simple $\D_X$-module whose restriction to $X\!\setminus\! D$ corresponds to $\chi$ via the Riemann--Hilbert correspondence. Consider the \'etale fundamental group $\pi_1^{\op{\acute{e}t}}(X\!\setminus\!D)$ of $X$, which is equal to the profinite completion of $\pi_1(X\!\setminus\!D)$. We write $\Lambda(\pi_1^{\op{\acute{e}t}}(X\!\setminus\!D))$  for the set of all isomorphism classes of (continuous) finite dimensional irreducible representations of $\pi_1^{\op{\acute{e}t}}(X\!\setminus\!D)$.

\begin{theorem}\label{thm:algdecomp}
The algebra $\alg_X(D)$ has a direct sum decomposition into indecomposable $\D_X \times \pi_1(X\!\setminus\!D)$-modules as:
\[\alg_X(D) = \bigoplus_{\chi\in \, \Lambda(\pi_1^{\op{\acute{e}t}}(X\!\setminus\!D))} \! \mc{S}^\chi(*D) \oo \chi.\]
\end{theorem}

\begin{proof}
As in the proof of Theorem \ref{thm:generator}, we can assume for simplicity that $D=\emptyset$. Take any $N\triangleleft\, \pi_1(X)$) of finite index. By Theorem \ref{thm:generator} and Maschke's Theorem, $\conn_{X}^{N}\subset \alg_X$ is a semi-simple regular connection with a $\D_X$-module decomposition 
\begin{equation}\label{eq:commute}
\conn_{X}^{N} = \bigoplus_{i=1}^k (\mc{S}^{\chi_i})^{\oplus \dim \chi_i}.
\end{equation}
Take any $g\in \pi_1(X)$. Since the action of $\D_X$ and $\pi_1(X)$ commute, $g$ induces a $\D_X$-module endomorphism of $\conn_{X}^{N}$. As the modules $\mc{S}^{\chi_i}$ are simple, by Schur's lemma we get a $\D_X \times \pi_1(X)$ decomposition (as we vary $g\in \pi_1(X)$) of the form
\[\conn_{X}^{N} = \bigoplus_{i=1}^k \mc{S}^{\chi_i} \oo V_i,\]
where $V_i$ is a representation of $\pi_1(X)$ with $\dim V_i = \dim \chi_i$. By construction, functions generating a $\D_X$-module isomorphic to $\mc{S}^{\chi_i}$ have monodromy $\chi_i$. Therefore, $V_i = \chi_i$ and $\conn_{X}^{N}$ is a semi-simple $\D_X \times \pi_1(X)$-module. By (\ref{eq:algfin}) $\alg_X$ is then a semi-simple $\D_X \times \pi_1(X)$-module with the required decomposition.
\end{proof}

Note that the above proof shows that the natural map $\bb{C}[\pi_1(X\!\setminus\!D)/N] \xrightarrow{\,\cong\,} \End_{\D_X}(\conn_X^N(D))$ is an algebra isomorphism. The semi-simplicity of $\alg_X$ as a $\D_X$-module follows also from the Decomposition Theorem (see Theorem \cite[Theorem 8.2.26]{htt}) applied to a map $p_1$ in a construction similar to (\ref{eq:riemann}).

Based on the result above, we give a procedure to compute (part of) the Bernstein--Sato polynomial of an algebraic function, by reducing the calculation to (a single copy of) each indecomposable $\D_X$-module.

\begin{prop}\label{prop:bfunalg}
Let $h$ be an algebraic function with $\Sing \D_X h =D$ and Galois group $H$. Write 
\[h = \sum_{\chi\in \, \Lambda(H)} h_\chi, \quad \mbox{ where } \, h_\chi \in \Gamma(X, \, \mc{S}^\chi(*D) \oo \chi) \,\, \bigcap \,\, \bb{C}[H] \cdot h. \]
\begin{itemize}
\item[(a)] For some $k\in \bb{N}$ we have $ \underset{\chi\in \, \Lambda(H)}{\op{lcm}} b_{h_\chi,D}(s) \,\,\, \vert  \,\,\, b_{h}(s) \,\,\, \vert \,\,\,\,  \underset{\chi\in \, \Lambda(H)}{\op{lcm}} b_{h_\chi,D}(s) \cdots b_{h_\chi,D}(s+k)$.
\item[(b)] Fix $\chi \in \Lambda(H)$. There exists $h_{\chi}' \in  \Gamma(X, \, \mc{S}^\chi(*D))$ such that   $h_{\chi}'\oo v \in \bb{C}[H] \cdot h_\chi$, with $v\in \chi$. For such a function $h_{\chi}'$, we have $b_{h_{\chi}',D}(s) \,\, \vert \,\, b_{h_\chi,D}(s)$.
\end{itemize}
\end{prop}

\begin{proof}
Part (a) follows readily from Lemmata \ref{lem:bsum}, \ref{lem:bsum2} and Theorem \ref{thm:algdecomp}. For part (b), let $h_\chi=\sum_{i=1}^n w_i \oo v_i$ with $v_i\in \chi$, and write $W= \op{span}\{w_1,\dots,w_n\} \subset \Gamma(X, \, \mc{S}^\chi(*D))$. Let $Z=\bb{C}[H] \cdot h_\chi \subset W \oo \chi$. By Schur's lemma, we must have $Z = W' \oo \chi$ as an $H$-representation with $\emptyset \neq W'\subset W$ and $H$ acting trivially on $V$. Then we can pick any non-zero $h'_\chi  \oo v \in W' \oo \chi$. Similarly, we see by Schur's lemma that $h_\chi \in \bb{C}[H] \cdot h$. Since the action of $H$ and $\D_X$ commute, clearly $b_{z,D}(s) \, \vert \, b_{h_\chi,D}(s)$ for any $z \in Z$.
\end{proof}

In particular, by the above and Lemma \ref{lem:bgen} we obtain that for algebraic functions along a fixed hypersurface $D$ and with a fixed Galois group $\pi_1(X\!\setminus\!D) \to H$, there are only a finite number of possibe roots for their Bernstein--Sato polynomials modulo the integers.
Also, note that the decomposition in Proposition \ref{prop:bfunalg} takes place in the Galois extension of $h$.

\section{$G$-finite functions}\label{sec:gfin}

Throughout $G$ stands for a connected affine algebraic group, acting algebraically on a connected smooth complex algebraic variety $X$.

\subsection{Equivariant $\D$-modules} \label{sec:equivd}

A \defi{rational representation} of $G$ is a vector space on which $G$ acts by linear transformations such that any of its elements is contained in a finite dimensional $G$-stable subspace on which $G$ acts algebraically.

A $\D_X$-module $\mc{M}$ is \defi{equivariant} if we have a $\D_{G\times X}$-isomorphism $p^*\mc{M} \rightarrow m^*\mc{M}$, where $p: G\times X\to X$ is the projection and $m: G\times X\to X$ the action map, satisfying the usual compatibility conditions (see \cite[Definition 11.5.2]{htt}). 

We denote by $\mf{g}$ the Lie algebra of $G$. Differentiating the $G$-action on $X$ we get a map from $\mf{g}$ to space of vector fields on $X$, which in turn yields a map $\mf{g} \to \D_X$. When $X$ is affine, equivariance of a $\D_X$-module means that the $\mf{g}$-action induced by latter map can be integrated to a rational $G$-action. The category $\opMod_G(\D_X)$ (resp. $\opmod_G(\D_X)$) of equivariant $\D$-modules is a full subcategory of the category $\opMod(\D_X)$ (resp. $\opmod(\D_X)$) of all $\mc{O}_X$-quasi-coherent (resp. coherent) $\D$-modules, closed under taking subquotients. Equivariance is a functorial property, preserved along equivariant maps, duality etc.

\begin{theorem}\label{thm:orbit}
Let $G/H$ be a homogeneous space. Then any equivariant coherent $\D_{G/H}$-module is regular holonomic, and the category $\opmod_G(\D_{G/H})$ is equivalent to the category of finite dimensional representations of the finite group $\Gamma = H/H^0$ (where $H^0$ denotes the connected component of $H$ containing the identity).
\end{theorem}

Let us construct the equivariant connections on $\D_{G/H}$ explicitly. Denote by $\pi: G \to G/H$ the principal fibration. For a finite dimensional representation $V$ of $\Gamma$, we put the corresponding connection $\mc{S}^V$ to have $\O_{G/H}=(\pi_*\O_G)^H$-structure $(\pi_*(\O_G \oo V^*))^H$, whose spectrum is the homogeneous vector bundle $G\times_H V$. It inherits a $\D_{G/H}$-module action from its $(\pi_*\D_{G})^H$-action, through the surjective map
\begin{equation}\label{eq:invdiff}
(\pi_*\D_{G})^H \longrightarrow \D_{G/H}.
\end{equation}
This descends to an action on $\mc{S}^V=(\pi_*(\O_G \oo V^*))^H$, as $(\pi_*\D_{G})^H$ is locally isomorphic to $\D_{G/H} \oo U\mathfrak{h}$, and $U\mathfrak{h}$ acts trivially on the fibers $V^*$ (here $U\mathfrak{h}$ is the universal enveloping algebra of $\mathfrak{h}$).

When $G$ acts on $X$ with finitely many orbits as in the cases (1)--(6) in Section \ref{sec:class}, every module in $\opmod_G(\D_X)$ is regular and holonomic \cite[Theorem~11.6.1]{htt}, and the category $\opmod_G(\D_X)$ is equivalent to the category of finitely generated modules over a finite dimensional algebra (see \cite[Theorem 4.3]{vilonen} or  \cite[Theorem 3.4]{catdmod}). For more details on categories of equivariant $\D$-modules, cf. \cite{catdmod}.

A special class of objects in $\opmod_G(\D_X)$ come from local cohomology functors $\mc{H}^i_Z(\bullet)$, for $Z$ a $G$-stable (locally) closed subset in $X$. Namely, for any $i\geq 0$ and $\mc{M}\in\opMod_G(\D_X)$ we denote the $i$-th local cohomology module of $\mc{M}$ with support in $Z$ by $\mc{H}^i_Z(\mc{M})$, which is an element of the category $\opMod_G^Z(\D_X)$ of equivariant $\D_X$-modules supported in $Z$.

Given a rational representation $V$ of $G$, we can induce a corresponding $G$-equivariant $\D_X$-module by (see \cite[Section 2]{catdmod})
\begin{equation}\label{eq:PV}
\mc{P}(V):= \D_X \oo_{U\lie} V.
\end{equation}
For an equivariant $\D$-module $\mc{M}$, we have
\begin{equation}\label{eq:PVHom}
\Hom_{\D_X} (\mc{P}(V), \mc{M}) = \Hom_{G}(V,\Gamma(X,\mc{M})).
\end{equation}

For the rest of this subsection, we assume that $G$ is a (connected) linearly reductive group.

Let $\LL(G)$ be the set of isomorphism classes of finite dimensional irreducible representations of the group $G$, which can be identified with the set of dominant integral weights. For $\ll\in\LL(G)$, we denote by $\ll^*$ the dual representation. Any rational $G$-module is semi-simple.

We call a rational $G$-representation $N$ is \defi{admissible} if each representation in $\LL(G)$ appears (up to isomorphism) with finite multiplicity in $N$. In this case $N$ decomposes as
\[
N\,\cong \,\bo_{\ll \in \LL(G)} \, V_\ll^{\,\oplus m_{\ll}(N)},
\]
with $m_{\ll}(N)\in\bb{N}$. We write $N_\lambda \cong  V_\ll^{\oplus m_{\ll}(N)}$ for the isotypical component corresponding to $\ll\in \LL(G)$.

In the case when $G$ acts on $X$ with finitely many orbits, the space of sections of any $\mc{M}\in \op{mod}_G(\D_X)$ has an admissible $G$-module structure \cite[Proposition 3.14]{catdmod}.

For an irreducible representation $V_\ll$  of dominant weight $\lambda$, we abbreviate $\mc{P}(\ll):=\mc{P}(V_\ll)$.

When $X$ is a $G$-module, another construction of objects in $\opmod_G(\D_X)$ comes from considering the (twisted) Fourier transform \cite[Section 4.3]{catdmod}. This functor gives a self-equivalence
\[\F : \opmod_G(\D_X) \xrightarrow{\sim} \opmod_G(\D_X).\]
For $\mc{M}\in \opmod_G(\D_X)$ we have as $G$-modules
\begin{equation}\label{eq:fourier}
\F(\mc{M}) \cong \mc{M}^* \cdot  \det X.
\end{equation}

By the same reasoning, we have the following result.

\begin{lemma}\label{lem:four}
For a finite dimensional rational representation $V$ of $G$, we have $\F(\mc{P}(V)) = \mc{P}(V^* \oo \det X)$.
\end{lemma}

\subsection{$G$-finite functions}

We define the following class of functions.

\begin{definition}
Let $\Omega\subset X$ be a domain. An analytic function $f\in \holo(\Omega)$ is called $\lie$-\defi{finite}, if the space  $U\lie \cdot f\subset \holo(\Omega)$ is finite dimensional. If this space can be integrated to a rational $G$-action, then the function $f$ is called $G$-\defi{finite}.
\end{definition}

We note that the $G$-module action on a $G$-finite function is not necessarily the same globally as the one on polynomials given by $(g\cdot f)(x) =f(gx)$, but only locally, simply because it might not be possible to extend its domain of definition to a $G$-saturated open subset. In the next statement we work exclusively with the analytic topology.

\begin{prop}\label{prop:analcont}
Consider a $G$-finite function $f$ on a domain $\Omega\subset X$. Let $\gamma: [0,1] \to G$ be a continuous path with $\gamma(0)=1$, and take the path $\gamma_{x_0}:[0,1] \to X$ given by $t\mapsto \gamma(t) \cdot x_0$, for some fixed $x_0 \in \Omega$. Then $f$ has an analytic continuation $f_t$ along $\gamma_{x_0}$. Moreover, if $\gamma(1)=1$ then $f_0 = f_1$ (in a neighborhood of $x_0$).
\end{prop}

\begin{proof}
We can take open neighborhoods $V\subset \Omega$ with $x_0 \in V$ and $W\subset G$ with $1\in W$, such that we have $F(gv) = (g^{-1} \cdot F)(v)$, for all $g\in W,v\in V$, and $F\in \, U\lie \cdot f$ (since the latter space is finite dimensional). For $0\leq t \leq 1$, we write $g_t = \gamma(t)$, put $V_t = g_t \cdot V \, \subset X$ and define the analytic function $f_t:V_t\to \bb{C}$ by $f_t(g_t \cdot v) = (g_t^{-1} \cdot f)(v)$, for $v\in V$. 

We now check that this construction is a well-defined continuation. Fix $t \in [0,1]$. There exists some $\varepsilon >0$ such that for all $t' \in [0,1]$ with $|t-t'|<\varepsilon$, we have $g_{t'} \in  g_t \cdot W$ and $\gamma_{x_0}(t') \in V_t$. Take $y \in V_t \cap V_{t'}$, and write $y=g_t \cdot v = g_{t'} \cdot v'$ for some $v,v' \in V$. Then the element $g'=g_t^{-1} \cdot g_{t'}$ lies in $W$ and $v=g' \cdot  v'$. We have
\[f_t(y)=f_t(g_t v) = (g_t^{-1} \cdot f) (v) = (g_t^{-1} \cdot f)(g' \cdot  v') = (g_{t'}^{-1} \cdot f)(v') =f_{t'}(g_{t'} v') = f_{t'}(y).\]
Thus, the continuation is well-defined. Clearly, if $\gamma$ is closed then $f_1 = f_0 = f_{|V}$.
\end{proof}

In a similar vein, we have the following characterization of $G$-finite functions using $\D$-modules.

\begin{lemma}\label{lem:gfindmod}
An analytic function $f$ is $G$-finite if and only if $\D_X f$ is a $G$-equivariant $\D$-module.
\end{lemma}

\begin{proof}
If $\mc{M}:=\D_X f$ is $G$-equivariant, then $f$ is $G$-finite since the space of sections of $\mc{M}$ is a rational $G$-module. Conversely, assume that $f$ is $G$-finite defined on some domain $\Omega$, and let $V$ be a finite dimensional $G$-stable subspace of $\holo(\Omega)$ containing $f$. By (\ref{eq:PVHom}) we have a surjection $\mc{P}(V)\twoheadrightarrow \mc{M}$. Since $\mc{P}(V)$ is equivariant, so is its quotient $\mc{M}$.
\end{proof}

We want to understand when $G$-finite functions are holonomic. It turns out that this is automatic when $X$ has a dense orbit (see Corollary \ref{cor:gfinalg}). We have the following converse statement.

\begin{prop}\label{prop:negative}
Assume that $X$ does not have a dense $G$-orbit. Then there exists a $\lie$-invariant analytic function that is not holonomic.
\end{prop}

\begin{proof}
By a theorem of Rosenlicht (see \cite[Section 2.3]{popvin}), the assumption implies that there exists a non-constant $G$-invariant rational function $f\in \bb{C}(X)$. Consider the $\lie$-invariant holomorphic function (on a domain)
\[h:=\dfrac{1}{\sin \circ \, f}.\] 
We show that $\mc{M}=\D_X h$ is not a holonomic $\D$-module. Assume by contradiction that it is, then $Z=\Sing \mc{M}$ is closed algebraic subvariety of $X$ with $\codim_X Z \leq 1$. Clearly, $Z$ contains all the level sets $f^{-1}(k\pi)$, with $k\in \bb{Z}$. These form an infinite number of disjoint hypersurfaces, contradicting that $Z$ is algebraic.
\end{proof}

\begin{remark}
According to a theorem of Luna \cite{luna}, when $G$ is reductive and $X$ affine, essentially all $G$-invariant holonomic functions can be constructed as above. Namely, they can be obtained as compositions of holomorphic functions with $G$-invariant regular functions. In fact, the statement holds for (global) $G$-finite functions as well (see Lemma \ref{lem:globgfin}).
\end{remark}

We proceed further with some topological considerations. For a fixed point $x_0 \in X$, consider the morphism $\psi_{X,x_0}: \pi_1(G,1) \to \pi_1(X,x_0)$ given by $\gamma \mapsto \gamma_{x_0}$ as in Proposition \ref{prop:analcont}. It is easy to see that the image of $\psi_{X,x_0}$ is a central subgroup of $\pi_1(X,x_0)$ \, (for example, consider the map of fundamental groups induced by the multiplication map $G\times X \to X$).  Let $\pi_1^G(X,x_0)$ denote their quotient, which we call the \defi{equivariant fundamental group} of $X$. Since $X$ is connected, it follows that this does not depend on the choice of the point $x_0 \in X$, so we may suppress it in the notation (in particular, if the action of $G$ on $X$ has a fixed point, then $\pi_1^G(X)  = \pi_1(X)$). We have an exact sequence
\begin{equation}\label{eq:eqfun}
\pi_1(G) \xrightarrow{\,\,\psi_X\,\,} \pi_1(X) \to \pi_1^G(X)  \to 1.
\end{equation}

Fix a domain $\Omega\subset X$ with embedding $i:\Omega\to X$ and consider again $i_*^{an} \mc{O}_\Omega^{an}$ as a $\D_X$-module. We define $\Gf_{\Omega,X}$ to be the \defi{sheaf of $G$-finite functions} on $\Omega$ by taking it to be the sum of all equivariant coherent $\D_X$-submodules of $i_*^{an} \mc{O}_\Omega^{an}$. By Lemma \ref{lem:gfindmod}, the space of sections $\Gamma(X, \Gf_{\Omega,X})$ is precisely the algebra of $G$-finite functions on $\Omega$. Furthermore, based on Proposition \ref{prop:analcont} one can see that the monodromy action on a section in $\Gf_{\Omega,X}(U)$ on some open $U \subset X$ factors through $\pi_1(U) \to \pi_1(X) \to \pi_1^G(X)$.

Next, we see that the solutions to equivariant $\D$-modules are $G$-finite functions (cf. Proposition \ref{prop:connsheaf}).

\begin{lemma}\label{lem:gfinsol}
$\Gf_{\Omega,X}$ is an $\mc{O}_X$-quasi-coherent sheaf of equivariant $\mc{D}_X$-algebras. Furthermore, for any $\mc{M} \in \opMod_G(\D_X)$, we have a natural isomorphism $\Sol_V(\mc{M}) \cong \Hom_{\D_X}(\mc{M},\Gf_{\Omega,X}).$ 
\end{lemma}

In the case of holonomic functions of Nilsson class, $G$-finiteness can be checked using monodromy.

\begin{prop}\label{prop:gfinreg}
Let $D\subset X$ be a $G$-stable hypersurface, $U=X\setminus D$ and $\Omega \subset U$ a simply-connected domain. Then
\[\Gf_{\Omega,X} \,\, \bigcap \,\, \conn_X(D) \,\, = \,\, \conn_X^{\op{im} \psi_{U}}(D) \quad \subset i_*^{an} \mc{O}_\Omega^{an}.\]
\end{prop}

\begin{proof}
Let $j:U\to X$ be the open embedding. Since $j_*$ preserves equivariance, we can reduce to the case $D=\emptyset$. Then the statement follows readily from the Riemann--Hilbert correspondence, as a regular connection is equivariant if and only if it is an equivariant local system, which is the case if and only if it descends to a $\pi_1^G(X)$-representation.
\end{proof}

For the rest of this subsection, we let $X$ be affine with $G$ reductive. Denote by $X/\!/G=\op{Spec}(\bb{C}[X]^G)$ the affine GIT quotient, and write $p: X \to X/\!/G$ for the induced surjective quotient map. The next result follows as in  \cite[Corollary 6.9]{schwarz}.

\begin{lemma}\label{lem:globgfin}
For any open Stein subset $Q\subset (X/\!/G)^{an}$ we have 
\[\Gf_{p^{-1}(Q),X} \, = \,\bb{C}[X] \, \oo_{\bb{C}[X]^G} \, \mc{O}_{X/\!/G}^{an}(Q).\]
\end{lemma}

Along the same principle, we have the following result on $G$-finite holonomic functions.

\begin{prop}\label{prop:hologfin}
Let $x \in X$ with $G_x$ trivial and $G\cdot x$ closed.  Take $f_1,\dots, f_n \in \C[X]^G$ algebraically independent (where $n=\dim X/\!/G$), and let $h$ be a $G$-finite holonomic function on some domain of $X$. Then we can write 
\[h= \sum_{i=1}^k  \, p_i \cdot h_i(f_1, \dots, f_n), \]
with $p_i\in \bb{C}[U]$, for a $G$-stable affine neighborhood $U$ of $x$ with $U=p^{-1}(p(U))$, and $h_i$ holonomic on a domain of $\bb{C}^n$, for all $i=1,\dots,k$.
\end{prop}

\begin{proof}
Put $\mc{M}=\D_{X_0} h$. By Luna's slice theorem \cite{luna2}, there is an open affine $p$-saturated neighborhood $U \subset X$ of $x$, and the restriction $p_0$ of $p$ to $U$ is a principal bundle. Therefore, by descent for $\D$-modules we have $p^*_0 \,( (p_{0*} \mc{M})^G) = \mc{M}$. This, together with the argument as in Lemma \ref{lem:etale}, yields the required decomposition.
\end{proof}

Note that for a simply-connected Stein open subset $Q\subset p(U \setminus \Sing \D_X h) \subset (X/\!/G)^{an}$, the decomposition above is in accordance with Lemma \ref{lem:globgfin}, as in this case $\pi_1^G(p_0^{-1}(Q))=1$, therefore $h$ can be extended to $p^{-1}_0(Q)$ by Proposition \ref{prop:analcont}. For a sample application of the result above, see \cite[Theorem 3.4]{stat}.

The map $p$ induces an algebra map $\D_X^G \to \D_{X/\!/G}$ from the ring of invariant operators on $X$ to the ring of differential operators on $X/\!/G$. In the following statement, reductivity of $G$ is only used to ensure that $X/\!/G$ is an algebraic variety.

\begin{prop}\label{prop:invalg}
The quotient map $p$ induces an isomorphism $\,p^*: \alg_{X/\!/G}\xrightarrow{\,\cong\,} (\alg_X)^{\lie}\,$ of $\D_X^G$-algebras.
\end{prop}

\begin{proof}
It is easy to see that pulling back by $p$ gives a well-defined $\D_X^G$-algebra monomorphism. To see surjectivity, pick $y \in (\alg_X)^{\lie}$. Let $f$ be the minimal monic polynomial of $y$ in the Galois extension of $y$. Since the action of $\lie$ commutes with the Galois (monodromy) action, all the roots of $f$ are $\lie$-invariant. Thus, by Vieta's formulas the coefficients of $f$ are in $(\bb{C}(X) \cap \alg_X)^\lie = \bb{C}[X]^G$ (see Theorem \ref{thm:algdecomp}).
\end{proof}

\begin{lemma}\label{lem:git}
Let $N \triangleleft \pi_1(X)$ be such that $\op{im} \psi_X \subset N$,and assume that $\opmod(\pi_1(X)/N)$ has a generator. Then $\conn_X^N$ is an admissible $G$-representation if and only if $X$ has a unique closed $G$-orbit.
\end{lemma}

\begin{proof}
By Proposition \ref{prop:gfinreg}, we have $\conn_X^N \subset \Gf_{\Omega,X}$, for a simply-connected domain $\Omega \subset X$. By a standard GIT argument, if $X$ does not have a unique closed $G$-orbit then $\C[X]^G \subset \conn_X^N$ is infinite dimensional.

Conversely, if $X$ has a unique cloed $G$-orbit then any finitely generated $G-\C[X]$-module is admissible \cite[Theorem 3.25]{popvin}. We conclude by Theorem \ref{thm:generator}.
\end{proof}

\subsection{Binary forms} \label{sec:bin0}

Here we exhibit some of the concepts introduced with applications to the space of binary forms. We use some of the results obtained in Section \ref{sec:class} for the cubic case. 

In this section $X= \Sym^n W$ with $\dim W=2$ and $n > 1$. The group $\GL_2$ acts on $X$ with kernel $K=\{\omega I_2 \, \mid \, \omega^n=1 \}$, hence we put $G=\GL_2/K$. We let $x,y$ be a basis of $W$, and $x^n,x^{n-1}y, \dots, x y^{n-1}, y^n$ a basis of $X$, with respective coordinates $x_n, x_{n-1}, \dots, x_1, x_0$, so that we identify $\C[X]=\C[x_0,\dots, x_n]$. Throughout we work with the convention that the highest weights in $\C[X]$ are non-negative (i.e. $W=(\bb{C}^2)^*$ as a representation of $\GL_2$).

We denote by $r=r(x_0,\dots,x_n)$ an algebraic function (on some domain) that satisfies
\begin{equation}\label{eq:root0}
x_n \cdot r^n +  x_{n-1} \cdot r^{n-1} + \dots  + x_1 \cdot r + x_0 =0.
\end{equation}
Throughout $f \in \bb{C}[X]^{\SL_2}$ denotes the discriminant of the polynomial above, with $\deg f = 2n-2$. It is easy to see that $\Sing \D_X r$ is defined by $x_n \cdot f$. In particular, $r$ is not $G$-finite! Furthermore, the monodromy representation induced by the Galois group $S_n$ is not irreducible. Therefore, for our purposes it is more natural to consider a different function instead.

The action of $G$ on $X$ gives the Lie algebra map $\lie \to \D_X$, and we pick the following basis of its image:
\[g_{11}=x_1 \partial_1 + 2 x_2 \partial_2 + \dots + n x_n \partial_n, \quad     g_{12}=n x_0 \partial_1 + (n-1) x_1 \partial_2 + \dots + x_{n-1} \partial_n, \]
\[g_{21}=x_1 \partial_0 + 2 x_2 \partial_1 + \dots + n x_n \partial_{n-1}, \quad 	g_{22}=n x_0 \partial_0+(n-1) x_1\partial_1 + \dots + x_{n-1} \partial_{n-1}.\]

We denote by $Z \subset X$ the Veronese cone, which is the $G$-stable closed subvariety given as the image of the map $W \to X$, $w \mapsto w^d$. Let $D\subset X$ be reduced divisor defined by $f$, and $D_0$ its smooth locus.

\begin{theorem}\label{thm:bin}
Put $h=x_{n-1}+n x_n \cdot r$. Then $h$ is $G$-finite of highest $\GL_2$-weight $(n-1,1)$, $\D_X \cdot h$ is a simple (equivariant) $\D_X$-module with $\charC(\D_X h)= (n-1) \cdot [T^*_X X] + [\ol{T^*_{D_0}X}]$, and $\mc{F}(\D_X h) \cong \mc{H}^{n-1}_Z(\mc{O}_X)$. Moreover, the following operators generate $\Ann(h)$:
\[g_{11}-n+1, \, g_{12}^{\,n-1}, \,  g_{21}, \, g_{22}-1, (\partial_i \partial_{j+1} - \partial_{i+1} \partial_{j})^{1+\delta_{j,n-1}}, \, \mbox{ for } 0\leq i<j \leq n-1\mbox{ (}\delta \mbox{ is the Kronecker delta)}.\]
\end{theorem}

\begin{proof}
The monodromy of $h$ is irreducible, corresponding to the standard representation of $S_n$. Clearly, $\Sing \D_X h \subset \Sing \D_X r  = D \cup \{x_n = 0\}$. As $r$ is homogeneous with respect to scaling (i.e. annihilated by $(g_{11}+g_{22})/n$), it is easy to see that it has (locally) trivial monodromy around the line $x_n = 0$ (see also Proposition \ref{prop:analcont}). From the inequality $|x_n r| \leq |x_n|+ \max\{|x_0|,\dots, |x_{n-1}|\}$ we see that $x_n r$ is locally bounded around this line.
Hence, for any root $r$ of (\ref{eq:root0}), $x_n \cdot r$ can be extended locally to the line $x_n = 0$ by Riemann's theorem on removable singularities. By Corollary \ref{cor:singtor}, we obtain $\Sing \D_X h = D$.

Let $\mc{J}\subset \D_X$ denote the ideal generated by the operators given in the statement. For homogeneity reasons, $g_{11}-n+1$ and $g_{22}-1$ annihilate $h$ (see \cite[p. 284]{mayr}). This implies that $\D_X h$ is a $T$-equivariant $\D_X$-module, where $T \subset \GL_2$ denotes the torus of diagonal matrices. Since the induced map $\pi_1(T) \to \pi_1(\GL_2)$ is surjective and $\Sing \D_X h$ is $G$-stable, we deduce by Proposition \ref{prop:gfinreg} that $\D_X h$ is actually $\GL_2$-equivariant, and hence also $G$-equivariant. Therefore, by Lemma \ref{lem:gfindmod} $h$ is a $G$-finite function. 

For the claim on highest weight, we show that $g_{21} \cdot h =0$. For this, we note that $G$ acts on $r$ via fractional transformations as it does on $x/y$. As the operator $g_{21}$ is induced by the linear transformation $x \mapsto y,\, y\mapsto 0$, we obtain $g_{21} \cdot r = -1$, and so $g_{21}\cdot h =0$.

Using (\ref{eq:PVHom}), we have a surjection $\mc{P}((n-1,1)) \to \D_X h$, which implies that $h$ is also annihilated by $g_{12}^{\,n-1}$ (see \cite[Section 2]{catdmod}). The rest of the operators $(\partial_i \partial_{j+1} - \partial_{i+1} \partial_{j})^{1+\delta_{j,n-1}}$ are easily seen to annihilate $h$ using the annihilating operators of $r$ given in \cite[p. 284]{mayr}. Therefore, we have $\mc{J}\subset \Ann(h)$.

Put $\mc{M}=\mc{F}(\D_X h)$. Then $\mc{M}$ is a $G$-equivariant holonomic $\D_X$-module. Since $J\subset \Ann(h)$, we have $\Supp \mc{M} \subset Z$. We note that $Z = O \cup \{0\}$, with $O$ being the highest weight orbit of $X$, and the $G$-stabilizer of $O$ is connected (in contrast to the setting in \cite{claudiu5}, since $G=\GL_2/K$). As seen in Section \ref{sec:equivd}, this implies that the category $\opmod_G^Z(\D_X)$ has $2$ simple objects, $E=\mc{F}(\mc{O}_X)$ and $\mc{M}_O$, the latter being the simple $\D_X$-module corresponding to the trivial local system on $O$. As $\mc{M}_O$ has no non-trivial self-extensions in  $\opmod_G^Z(\D_X)$, and $\D_X h$ is indecomposable (see Lemma \ref{lem:simpmon} (c)) with no simple composition factors isomorphic to $\mc{O}_X$, we deduce that $\mc{M}=\mc{M}_O$. As $\rank \D_X h = n-1$, it follows that $\charC(\D_X h)= (n-1) \cdot [T^*_X X] + [\ol{T^*_{D_0}X}]$.

As in Theorem \ref{thm:algdecomp}, it is easy to see that there are no non-trivial extensions between $\mc{O}_X$ and $\D_X h$. Hence, the category $\opmod_G^Z(\D_X)$ is semi-simple. Since $\mc{H}^{n-1}_Z(\mc{O}_X)$ is the injective hull of $\mc{M}_O$ in $\opmod_G^Z(\D_X)$ (cf. \cite[Lemma 3.11]{catdmod}), we deduce that $\mc{M}=\mc{H}^{n-1}_Z(\mc{O}_X)$.

We are left to show that $\mc{J}=\Ann(h)$. As seen above, we have a surjection $\mc{P}((n-1,1)) \to \D_X/\mc{J}$, which shows that $\D_X/\mc{J}$ is $G$-equivariant. Since $\Supp \mc{F}(\D_X/\mc{J}) \subset Z$ and $\opmod_G^Z(\D_X)$ is semi-simple, we have $\D_X/\mc{J} \cong (\O_X)^{\oplus a} \oplus (\D_X h)^{\oplus b}$, with $a\geq 0, b\geq 1$. We have by (\ref{eq:PVHom})
\[a=\dim \Hom_{\D_X}(\D_X/\mc{J}, \mc{O}_X) \leq \dim \Hom_{\D_X}(\mc{P}((n-1,1)), \mc{O}_X) = m_{(n-1,1)}\left(\Sym \Sym^n \C^2\right)= 0.\]
One can conclude by an explicit rank calculation that we must have $b=1$. Alternatively, by (\ref{eq:fourier}) we have 
\[b=\dim \Hom_{\D_X}(\D_X/\mc{J}, \D_X h) \leq \dim \Hom_{\D_X}(\mc{P}((n-1,1)), \D_X h) = m_{(n-1,1)}\left(\D_X h\right)= m_\ll(M),\]
where $\ll = \big(n(n+1)/2-1, \, n(n-1)/2+1\big)$. We now see using \cite[Theorem 1.2]{claudiu5} that $m_\ll(M)=1$.
\end{proof}

We note that the above result gives an alternative approach to \cite[Theorem 4.1]{claudiu5} in this case. In addition, we have determined an explicit presentation for $\mc{H}^{n-1}_Z(\mc{O}_X)$ as well as its characteristic cycle. By \cite[Theorem 1.2]{claudiu5} and (\ref{eq:fourier}), we have a description of the $G$-module structure of $\D_X h$. In particular, $\D_X h$ is an admissible representation, which does not follow directly from Lemma \ref{lem:git}. In fact, by classical invariant theory it is known that if $n\geq 4$ then $\C[X]^{\SL_2}$ has more than $1$ generator, so we see that in fact $(\D_X h)(*D)$ is not admissible.

A basic problem is to determine $b_h(s)$, which is also closely related to \cite[Theorem 7.1]{opdam} and \cite{asrob}. We know that the roots are rational, and conjecture that they are strictly negative, as suggested by Theorem \ref{thm:bin}. We find some \lq\lq obvious" roots, which is sufficient for our purposes in Section \ref{sec:class}.

\begin{prop}\label{prop:binroot}
For $h=x_{n-1}+n x_n \cdot r$ and $n\geq 3$, we have $b_h(-1)=b_h(-3/2) =0$.
\end{prop}

\begin{proof}
Since $n\geq 3$, we can choose a point $\ul{a}=(a_0,\dots, a_n) \in D \subset X$ with $a_n\neq 0$ such that (\ref{eq:root}) has a root $c \neq -a_{n-1}/(n a_n)$ of multiplicity one. Using the implicit function theorem on (\ref{eq:root}), we can find an analytic function $r$ on some domain of $X$ containing ${\ul{a}}$ such that it is a solution to (\ref{eq:root}) and $r(\ul{a})=c$. By Lemma \ref{lem:normalize}, we obtain $b_h(-1)=0$.

Using monodromy, the second claim is equivalent to $b_{h'}(-1)=0$, where $h'=x_n(r_1-r_2)/\sqrt{f}$. Choose a point $\ul{b}=(b_0,\dots,b_n) \in D$ such that $b_n\neq 0$ and (\ref{eq:root}) has roots $c_1,\dots,c_n$ with the property that $c_1=c_2$ and the rest of the roots are all pairwise distinct. We can find a simply-connected domain $\Omega \subset X$ containing $\ul{b}$ such that for $3 \leq i \leq n$ the roots $r_i$ of  (\ref{eq:root}) are all analytic with $r_i(\ul{b})=c_i$ and $x_n \neq 0$. The analytic functions $r_1 - r_2$ and $\sqrt{f}$ are well-defined on any simply-connected domain of $\Omega \setminus D$, where they have the same monodromy corresponding to the sign representation of $S_n$. Hence, $h'$ can be analytically continued to $\Omega \setminus D$. Since on $\Omega \setminus D$ we have $f=x_n^{2n-2} \prod_{i<j}(r_i-r_j)^2$, we see by Riemann's theorem on removable singularities that $h'$ can be analytically continued to the whole $\Omega$, and $h'(\ul{b})\neq 0$. The conclusion follows again by Lemma \ref{lem:normalize}. 
\end{proof}

\subsection{Multiplicity-free holonomic functions}\label{sec:multfree}

In this section $X$ is finite dimensional rational representation of a (connected) reductive group $G$. Here we provide a technique for computing Bernstein--Sato polynomials of $G$-finite functions. 

The following is an immediate consequence of Corollary \ref{cor:singtor}.

\begin{lemma}\label{lem:semi}
If $h$ is a $G$-finite holonomic function, then $\Sing (\D_{X} h)$ is a hypersurface in $X$ defined by a semi-invariant polynomial.
\end{lemma}

\begin{definition}\label{def:multfree}
Let $h \in \Gf_{\Omega,X}$ and $f = \prod_{i=1}^l f_i \in \C[X]$ the semi-invariant of weight $\sigma$ defining $\Sing (\D_X h)$, with $f_i$ irreducible ($1\leq i \leq l$). Assume that $U\lie\cdot h \, \cong V_\lambda$ for some $\lambda\in \Lambda(G)$. Then $h$ is \defi{multiplicity-free} if $m_{\lambda+k\sigma}(\D_X h) = 1$ for $k=\underset{1\leq i \leq l}{\max}\{\deg f_i\}$.
\end{definition}

A necessary condition for the above to hold is $m_{k\sigma}(\C[X])=1$, which happens when $f$ is multiplicity-free in the sense \cite[Section 1.2]{bub} (e.g. if $X$ is prehomogeneous). In particular, each $f_i$ must be homogeneous (see \cite[Proposition 4.3]{saki}). Since $G$ is reductive, there is a dual semi-invariant $f^*\in \C[X^*]$ of weight $\sigma^{-1}$ of the same degree, unique up to constant. In fact, we can choose a basis of $X$ such that the the image under $G\to \GL(X)$ is stable under conjugate transpose, in which case $f^*$ can be obtained from $f$ by taking the complex conjugates of the coefficients, and replacing the variables by the dual (partial) variables \cite[Section 4]{saki}. For $f_1,\dots, f_l$, we take respective dual semi-invariants $f^*_1,\dots, f^*_l\in \C[X^*]$.

\begin{theorem}\label{thm:bfunmult}
Assume that $h$ is a multiplicity-free holonomic function. Then
\begin{equation}\label{eq:good}
f^*(\partial) \cdot f^{s+1}h = b(s) \cdot f^s h,
\end{equation}
where $b(s)$ is a polynomial with $\deg b(s) = \deg f$ and $b_h(s) \,\vert\, b(s)$. If the holonomic $\D$-module $\D_X h$ is also regular, then $b_h(s)=b(s)$ (up to a non-zero constant factor).
\end{theorem}

\begin{proof}
Write $U\lie \cdot h \, \cong V_\lambda$ as in Definition \ref{def:multfree}. Fix an element $p\in \{1,\dots, l\}$, and let $\sigma_p$ be the character of $f_p$. First, we show that there is a $b_p(s) \in \C[s]$ such that
\begin{equation}\label{eq:bfac}
f_p^*(\partial) \cdot  f_p f^s h' = b_p(s) \cdot f^s h'.
\end{equation}
for all $h' \in V_\lambda$. Take a highest weight vector $h_\lambda \in V_\lambda$. Clearly, we have 
\[f_p^*(\partial) \cdot  f_p f^s h_\lambda = f^{s-\deg f_p} \cdot f_p \cdot Q(s),\]
where $Q(s) \in \C[s]\oo_{\C} (\D_X \cdot h)_{\lambda + \deg f_p \cdot \sigma - \sigma_p}$. By assumption, we have $(\D_X \cdot h)_{\lambda + \deg f_p \cdot \sigma}= f^{\deg f_p} \cdot V_\lambda$, and thus $f_p \cdot Q(s) = b_p(s) \cdot f^{\deg f_p} \cdot h_\lambda$, for some $b_p(s) \in \C[s]$. Applying $\lie$-translates and via linearity, we obtain (\ref{eq:bfac}) for any $h' \in V_\lambda$. This yields the functional equation (\ref{eq:good}) with $b(s) = b_1(s) \cdots b_l(s)$ and clearly $b_h(s) \,\vert\, b(s)$.

Next, we clearly have $\deg b(s) \leq \deg f$. To show that equality holds, note that the coefficient of $s^{\deg f}$ in $b(s)$ is the same as that in $b'(s)$ given by the equation $f^*(\partial) \cdot f^{s+1} = b'(s) \cdot f^s$ for the $b$-function of $f$, which is non-zero (see \cite[Section 4]{saki}).

Now assume that $\D_X h$ is regular holonomic. By Lemma \ref{lem:charc} we have $\Char( (\D_X h)_f ) = \Char (\C[X]_f)$. As in the proof of \cite[Corollary 2.5.10]{gyoja}
, we deduce that there is a point $y \in \Char ( (\D_X h)_f )$ such that the microdifferential operator corresponding to $f^*(\partial)$ is invertible near $y$. This implies that (up to constant)  we have $b_h(s) = b(s)$ (see also \cite[Section 5]{gyoja2}).
\end{proof}

An example of a multiplicity-free holonomic function is $h=f^{-1}$ for a multiplicity-free semi-invariant $f$ on a prehomogeneous vector space, when we recover the classical equation for the Bernstein--Sato polynomial of $f$ (see \cite[Lemma 1.6,1.7]{gyoja} and \cite[Corollary 2.5.10]{gyoja}).

We have the following notion of $b$-function of several variables (compare with \cite{sata}). For a multi-variable $\underline{s}=(s_1,\dots ,s_l)$, we let $\underline{f}^{\underline{s}}=\prod_{i=1}^l f_i^{s_i}$, and $\underline{f}^{*\underline{s}}=\prod_{i=1}^l f_i^{*s_i}$. Let $d_i = \deg f_i$. Using the analogoue of (\ref{eq:bfac}) used in the proof of Theorem \ref{thm:bfunmult}, we have the following equation.

\begin{prop}\label{prop:several}
Assume that $h$ is a multiplicity-free holonomic function. Then for any $l$-tuple $\underline{m}=(m_1,\dots,m_l)\in \bb{N}^l$ there is a polynomial $b_{h,\underline{m}}(\underline{s})$ of $l$ variables such that 
\[\underline{f}^{*\underline{m}}(\partial) \, \cdot \, \underline{f}^{\underline{s}+\underline{m}}\cdot h=b_{h,\underline{m}}(\underline{s}) \cdot \underline{f}^{\underline{m}}\cdot h.\]
Furthermore, each $b_{h,\underline{m}}(\underline{s})$ is a product of linear factors of the form  (up to a non-zero scalar)
\[b_{h,\underline{m}}(\underline{s}) = \, \prod_{i=1}^{\ul{d}\cdot \ul{m}} \, (\gamma_{i1} s_1 + \gamma_{i2} s_2 + \dots + \gamma_{il} s_l + \alpha_i),\]
where for all $i,j$ we have $\gamma_{ij} \in \bb{N}, \, \op{gcd}(\gamma_{i1},\dots, \gamma_{il})=1$, and $\alpha_i \in \bb{C}$. If $h$ is as in Proposition \ref{prop:rat} then furthermore $\alpha_i \in \bb{Q}$ for each $i$.
\end{prop}

The decomposition of $b_{h,\underline{m}}(\underline{s})$ can be deduced by an argument as in \cite[Sections 3 and 4]{sata}.

We note that an equation as above with $b(\underline{s})$ having linear factors exists for an arbitrary holonomic function $h$ by \cite[Th\'eor\`eme 2.1]{sabbah1} (see also \cite[Proposition 1.2]{sabbah2} and \cite{gyoja3}).

\begin{remark}\label{rem:slight}
The arguments above show that for the results to hold we can replace the multiplicity-free condition in Definition \ref{def:multfree} with the slightly weaker condition $m_{\lambda+\deg f_i \cdot \sigma -\sigma_i}(\D_X h)=1$ for all $i=1,\dots, l$.
\end{remark}

It is immediate to recover the $b$-function $b_{h}(s)$ of one variable from the $b$-function of $b_{h,\underline{m}}(\underline{s})$ of several variables (for $h$ multiplicity-free of Nilsson class). 

\section{$G$-finite functions in the presence of a dense orbit} \label{sec:dense}

We assume throughout this section that $X$ is a connected smooth variety with a dense orbit under the action of a connected affine algebraic group $G$. We denote the open $G$-orbit by $O$ with $O\cong G/H$. Write $X\setminus O= D \cup C$, where $D$ is a hypersurface, and $\codim_X C\geq 2$.

Let $\G:=H/H^0$ denote the component group of $O$, and $\Lambda(\Gamma)$ the (finite) set of (isomorphism classes of) irreducible representations of $\G$. By (\ref{eq:eqfun}), we have $\pi_1^G(O) = \pi^G_1(X\backslash D) = \Gamma$. By Theorem \ref{thm:orbit}, there is a 1-to-1 correspondence between the elements of $\Lambda(\Gamma)$ and the isomorphism classes of simple equivariant holonomic $\D_X$-modules with full support (moreover, torsion-free). We denote by $\mc{S}^\chi \in \opmod_G(\D_X)$ the simple (regular holonomic) $\D_X$-module corresponding to $\chi \in \Lambda(\Gamma)$.

\subsection{Basic results}\label{subsec:equiv} Our starting is the following observation.

\begin{prop}\label{prop:prehomrank}
Let $\mc{M}$ be a $G$-equivariant coherent $\D_X$-module. Then $\mc{M}$ has finite rank and $\mc{M}^w$ is a regular holonomic $G$-equivariant $\D_X$-module.
\end{prop}

\begin{proof}
Let $Z=X\setminus O$ and $j: O \to X$ the open embedding. Due to equivariance, $\Gamma_Z \mc{M}$ is the torsion subsheaf of $\mc{M}$. Hence, from the exact sequence
\[0 \to \Gamma_Z (\mc{M}) \to \mc{M} \xrightarrow{\alpha} j_*j^*(\mc{M}) \xrightarrow{\beta} \mc{H}^1_Z(\mc{M}) \to 0,\]
we see that $\mc{M}^w = \op{im} \alpha$ is a submodule of $j_*j^*(\mc{M})$. It is enough to show that the latter is regular holonomic. Note that $j^*(\mc{M})$ is an equivariant coherent $\D_O$-module, hence it is regular holonomic by Theorem \ref{thm:orbit}. Hence, so is $j_*j^*(\mc{M})$, and $\mc{M}$ has finite rank by Lemma \ref{lem:rankhol}.
\end{proof}

Take an arbitrary simply-connected domain $\Omega \subset X\setminus D$. We see below that $\Gf_{\Omega,X}$ does not depend on the choice of $\Omega$ (up to isomorphism), hence we simply put $\Gf_X:=\Gf_{\Omega,X}$.

\begin{corollary}\label{cor:gfinalg}
We have $\Gf_{X} = \conn_X^{\!\!\op{im} \psi_{X\!\setminus \! D}}\!(D) \,\,\,\subset \alg_X(D)$, with a decomposition into indecomposable $\D_X \times \Gamma$-modules
\[\Gf_{X} = \bigoplus_{\chi\in \Lambda(\Gamma)} \mc{S}^\chi(*D)\oo \chi.\]
In particular, $\Gf_{X}$ is a regular meromorphic connection. Furthermore, $\Gf_{X}$ is injective in $\opMod_G(\D_X)$ and
\[\dim \Hom_{\D_X}(\mc{M},\Gf_{X}) = \rank \mc{M},\]
for any $\mc{M} \in \opmod_G(\D_X)$.
\end{corollary}

\begin{proof}
Take a coherent submodule $\mc{M} \subset \Gf_{X}$. By Proposition \ref{prop:prehomrank}, $M$ is a regular holonomic torsion-free $\D$-module. Since $\Sing M$ is a $G$-stable divisor, we must have $\Sing M \subset D$. Thus, we have an inclusion $M\subset \conn_X(D)$. Since $M$ was an arbitrary coherent submodule of $\Gf_{X}$, we get $\Gf_{X} \subset \conn_X(D)$. By Proposition \ref{prop:gfinreg}, this implies that $\Gf_{X} = \conn_X^{\!\!\op{im} \psi_{X\!\setminus \! D}}\!(D)$, and this is contained in $\alg_X(D)$ since $\Gamma = \pi_1^G(X)$ is finite. The claim on the decomposition of $\Gf_{X}$ follows from Theorem \ref{thm:algdecomp}. Noticing from Proposition \ref{prop:prehomrank} that $\opMod_G(\D_X) \subset \opMod_{D,w}^{\op{im} \psi_{X\!\setminus\!D},rh}(\D_X)$, the rest follows from Proposition \ref{prop:connsheaf}.
\end{proof}

\begin{lemma}\label{lem:singsimp}
Let $\chi$ be a non-trivial irreducible representation of $\Gamma$. Then the hypersurface $\Sing \mc{S}^\chi$ is contained in $D$, and is non-empty if $X$ is simply-connected.
\end{lemma}

\begin{proof}
Since $\Sing \mc{S}^\chi$ is a $G$-stable divisor, and $\Sing \mc{S}^\chi \cap O = \emptyset$, the first claim is clear. Assuming that it is empty, we get by Proposition \ref{prop:prehomrank} that $\mc{S}^\chi$ is a regular connection on $X$. If $X$ is simply-connected, we get $\Sing \mc{S}^\chi \cong \O_X^{\oplus \dim \chi}$ by the Riemann--Hilbert correspondence, contradicting the non-triviality of $\chi$.
\end{proof}

The geometric counterpart of rank finiteness in Proposition \ref{prop:prehomrank} is as follows. Take the moment map
\[\mu: T^* X \to \lie^*.\]
Since $X$ has a dense $G$-orbit, the zero section $T_X^* X$ is an irreducible component of the zero fiber $\mu^{-1}(0)$. Moreover, the multiplicity of the scheme-theoretic zero fiber $\mu^{-1}(0)$ along $T_X^* X$ is one (see \cite[Lemma 3.12]{catdmod}). We have the following result, which follows as in the proof of  \cite[Proposition 3.14]{catdmod}.

\begin{lemma}\label{lem:rankrep}
Let $V$ be a finite dimensional representation of $G$. Then $\rank \mc{P}(V)\leq \dim V$.
\end{lemma}

The following is a basic result toward understanding the category $\op{mod}_G(\D_X)$.

\begin{prop}\label{prop:injproj}
For a simple $\D_X$-module $\mc{S}^\chi$ with $\chi \in \Lambda(\Gamma)$, the following holds in $\opmod_G(\D_X)$:
\begin{itemize}
\item[(a)] The module $\mc{I}^\chi:=\mc{S}^\chi(*D)$ is its injective envelope.
\item[(b)] The module $\mc{P}^\chi:=\bb{D} (\mc{S}^{\chi^*}(*D))$ is its projective cover, and $(\mc{P}^\chi)^w = \mc{S}^\chi$.
\end{itemize}
\end{prop}

\begin{proof}
The first part follows from \cite[Lemma 2.4]{bindmod} or Corollary \ref{cor:gfinalg}. Analogously, the module $\mc{P}^\chi=j_! j^* \mc{S}^{\chi^*}$ is the projective cover, where $j:X\setminus D \to X$ is the open embedding and $j_! = \bb{D} j_* \bb{D}$.

By construction, we have an exact sequence 
\[0\to \mc{K} \to j_! j^* \mc{S}^{\chi^*} \to \mc{S}^\chi \to 0.\]
for some $\mc{K} \in \opmod_G(\D_X)$ that has support in $X\setminus O$. This yields $(\mc{P}^\chi)^w = \mc{S}^\chi$.
\end{proof}

\subsection{The case when $G$ is reductive}\label{sec:red}

Here we assume additionally that $X$ is affine and $G$ is reductive. In particular, we identify $\Gf_{X}$ with the algebra of $G$-finite functions on $X$.

\begin{lemma}\label{lem:PGf}
For an irreducible representation $\ll$ of $G$, we have 
\[m_\ll(\Gf_{X}) = \rank \mc{P}(\ll) =\sum_{\chi\in\Lambda(\Gamma)} \dim \chi \cdot m_{\ll}(\mc{S}^\chi(*D)).\] 
In particular, any torsion-free module in $\opmod_G(\D_X)$ is an admissible $G$-representation.
\end{lemma}

\begin{proof}
By (\ref{eq:PVHom}) and Corollary \ref{cor:gfinalg} we have
\[m_\ll(\Gf_{X})=\Hom_\D(\mc{P}(\ll),\Gf_{X}) = \rank \mc{P}(\ll).\]
Lemma \ref{lem:rankrep}, this is finite. We get the second equality by Corollary \ref{cor:gfinalg} again. By Proposition \ref{prop:prehomrank}, an $\O_X$-torsion-free equivariant coherent $\D_X$-module can be realized as a submodule of a (finite) direct sum of copies of $\Gf_{X}$, hence it must be admissible as a representation of $G$ (see also Lemma \ref{lem:git}).
\end{proof}

We recall the explicit correspondence between representations of $\Gamma$ and connections on $O$ as described in the discussion after Theorem \ref{thm:orbit}.

\begin{lemma}\label{lem:peterweyl}
For any irreducible representation $\ll$ of $G$ and $\chi\in\Lambda(\Gamma)$, we have
\[m_{\ll}(\mc{S}^\chi(*D)) = m_{\chi}(V_\ll^{*H_0}).\]
\end{lemma}

\begin{proof}
We have the following decomposition as $G$-modules 
\[\mc{S}^\chi(*D) = H^0(G/H, (\pi_*(\O_G \oo \chi^*))^H) = (\C[G]\oo \chi^*)^H \cong \bigoplus_{\ll \in \Lambda(G)} V_\ll \oo (V_\ll^*\oo \chi^*)^H.\]
The claim now follows from $\dim (V_\ll^*\oo \chi^*)^H = m_{\chi}(V_\ll^{*H_0})$.
\end{proof}

\begin{definition}\label{def:witness}
An irreducible representation $\ll$ of $G$ is called a \defi{witness representation} for $\mc{S}^\chi$, if $m_\ll(\mc{S}^\chi) \neq 0$ and $\rank \mc{P}(\ll) = \dim \chi$.
\end{definition}

Recall that for a holonomic $\D$-module $\mc{M}$ and a simple $\D$-module $\mc{S}$, we denote by $[\mc{M}:\mc{S}]$ the multiplicity of $\mc{S}$ in a composition series of $\mc{M}$. The following clarifies the terminology introduced above.
 
\begin{lemma}\label{lem:witness}
Assume that $\ll$ is an irreducible representation of $G$. Then $\ll$ is a witness representation for $\mc{S}^\chi$ if and only if for any torsion-free equivariant coherent $\D$-module $\mc{M}$, we have $[\mc{M}:\mc{S}^\chi] = m_{\ll}(\mc{M})$.
\end{lemma}

\begin{proof}
By Lemma \ref{lem:PGf}, $\ll$ is a witness representation for $\mc{S}^\chi$ in and only if $m_\ll(\mc{S}^\chi)=m_\ll(\mc{S}^\chi(*D))=1$ and $m_\ll(\mc{S}^{\chi'}(*D))=0$, for all $\chi'\in \Lambda(\Gamma), \chi'\neq \chi$. Since $\mc{M}$ can be embedded into a direct sum of $\Gf_{X}$, we conclude by Theorem \ref{thm:algdecomp}.
\end{proof}

\begin{example}
There are no non-constant $\lie$-invariants in $\Gf_{X}$, hence the trivial representation is a witness representation for $\C[X]$. More generally, the weights of products of powers of semi-invariants as in Proposition \ref{prop:normal} are witness representations for the corresponding simple $\D$-modules.
\end{example}

The projective equivariant $\D_X$-modules $\mc{P}(\ll)$ have explicit presentations as a cyclic $\D$-modules \cite[Section 2]{catdmod}. Thus, the following gives a computable strategy for finding explicit presentations of the simples $\mc{S}^\chi$ using Weyl closure (compare with Proposition \ref{prop:injproj}).

\begin{lemma}\label{lem:witpres}
Assume $\ll$ is a witness representation for $\mc{S}^\chi$. Then $\mc{P}(\ll)^w=\mc{S}^\chi$.
\end{lemma}

\begin{proof}
By (\ref{eq:PVHom}), we have an exact sequence 
\[0 \to \mc{K} \to \mc{P}(\ll) \to \mc{S}^\chi\to 0.\]
for some $\mc{K}\in\opmod_G(\D_X)$. Since $\ll$ is a witness representation, $\rank \mc{P}(\ll)=\rank \mc{S}^\chi$. Hence, we must have $\rank \mc{K}=0$ so that $\mc{K}^w=0$. This gives $\mc{P}(\ll)^w=\mc{S}^\chi$.
\end{proof}

Besides finding presentations, witness representations are also ideal for computing Bernstein--Sato polynomials using the technique developed in Section \ref{sec:multfree}. For the rest of the section, $X$ is the affine space and $D$ is defined by a reduced semi-invariant $f$ of weight $\sigma$. 

To simplify notation, from now on for any $G$-finite function $h$ we put $b_h(s) := b_{h,D}(s)$ -- albeit $\Sing \D_X h$ can be smaller, see also Lemma \ref{lem:singsimp}. Note that we can always divide $h$ by a suitable power of $f$ so that the singular locus becomes $D$, and this produces only a shift in the Bernstein--Sato polynomial.

\begin{lemma}\label{lem:witmult}
Let $h \in \mc{S}^\chi$ such that $U\lie \cdot h \, \cong V_\lambda$ with $\ll$ a witness representation for $\mc{S}^{\chi}$. Then $h$ is a multiplicity-free holonomic function.
\end{lemma}

\begin{proof}
By Lemma \ref{lem:PGf}, we have $m_\ll(\mc{S}^\chi(*D))=1$, so $m_{\ll+k \sigma}(\mc{S}^\chi)=1$ for all $k\in \bb{N}$.
\end{proof}

When searching for (multiplicity-free) holonomic functions generating witness representations, the following is our main criterion.

\begin{prop}\label{prop:witcri}
Let $\ll\in \Lambda(G)$ and $\chi \in \Lambda(\Gamma)$. The following statements hold:
\begin{itemize}
\item[(a)] If $\ll$ is a witness representation for $\mc{S}^\chi$, then $(V_\ll^*)^{H_0} \cong \chi$ as $\Gamma$-modules.
\item[(b)] Conversely, if $(V_\ll^*)^{H_0} \cong \chi$ then a non-zero $h \in (\mc{S}^\chi(*D))_{\ll}$ is a multiplicity-free holonomic function. Also, if $\alpha$ is the largest integer root of $b_h(s)$, then $\ll + (\alpha+1) \sigma$ is a witness representation for $\mc{S}^\chi$.
\end{itemize}
\end{prop}

\begin{proof}
Assume that $\ll$ is a witness representation of $\mc{S}^\chi$ for some $k\in\bb{N}$. Then by Lemmata \ref{lem:witness} and \ref{lem:peterweyl} we have $m_\chi(V_\ll^{*H_0})=1$ and  $m_{\chi'}(V_\ll^{*H_0})=0$ for all irreducible representations $\chi' \neq \chi$ of $\Gamma$. This implies that we have $V_\ll^{*H_0} \cong \chi$.

Conversely, assume that $(V_\ll^*)^{H_0} \cong \chi$ as $\Gamma$-modules. By Lemma \ref{lem:peterweyl}, we have $m_\ll(\mc{S}^\chi(*D))=1$ and $m_\ll(\mc{S}^{\chi'}(*D))=0$ for all $\chi' \in\Lambda(\Gamma),\chi' \neq \chi$. By Lemma \ref{lem:bgen}, we have $f^{\alpha+1} \cdot h \in (\mc{S}^\chi)_{\ll+(\alpha+1)\sigma}$. In particular, $\ll+(\alpha+1)\sigma$ is a witness representation as seen in Lemma \ref{lem:witness}.
\end{proof}

The following result reveals a perfect symmetry between the roots of Bernstein--Sato polynomials (of several variables) with respect to duality.

\begin{theorem}\label{thm:bfourier}
Assume that $X\setminus O = D$ is a hypersurface, $f_1,\dots,f_l$ define its irreducible components with respective weights $\sigma_1, \dots, \sigma_l$, and $\ll$ is a witness representation for $\mc{S}^\chi$, for some $\chi \in \Lambda(\Gamma)$. We can write  $\det X = c_1 \sigma_1 + \dots + c_l \sigma_l$, with $2 c_i \in \bb{Z}$, and $\ll^* +p \cdot \sigma$ is a witness representation for $\mc{S}^{\chi^*}$, for some $p\in \bb{N}$. Moreover, for non-zero $h \in (\mc{S}^\chi)_\ll$ and $h^* \in (\mc{S}^{\chi^*})_{\ll^*}$, we have for any tuple $\underline{m}=(m_1,\dots,m_l)\in \bb{N}^l$
\[b_{h, \ul{m}}(\ul{s})  \,\, = \,\, \pm \, b_{h^*,\ul{m}}(-\ul{s}-\ul{m}+\ul{c}).\] 
\end{theorem}

\begin{proof}
Since $O$ is affine, by Matsushima's theorem \cite{matsu} the generic stabilizer $H$ is reductive. Hence, it follows from Lemma \ref{lem:peterweyl} that for any irreducible representation $\mu$ of $G$ and any $\psi \in \Lambda(\Gamma)$, we have
\begin{equation}\label{eq:multdual}
m_\mu(\mc{S}^\psi(*D)) = m_{\psi}(V_\mu^{*H_0}) = m_{\psi^*}(V_\mu^{H_0}) = m_{\mu^*}(\mc{S}^{\psi^*}(*D)).
\end{equation}
Since $\ll$ is a witness representation, this shows that $m_{\ll^*}(\mc{S}^{\chi^*}(*D))=1$, and $m_{\ll^*}(\mc{S}^{\psi}(*D))=0$ when $\psi \neq \chi^*$. Thus, the representation $\ll^*$ twisted by a suitable power $p$ of $\sigma$ is a witness representation for $\mc{S}^{\chi^*}$. By \cite[Proposition 8]{saki}, there is a semi-invariant rational function that has weight $\det^2 X$, i.e. we can write  $\det X = c_1 \sigma_1 + \dots + c_l \sigma_l$, with $2 c_i \in \bb{Z}$ (and such an expression is unique, e.g. \cite[Lemma 4.4]{saki}). As each $\sigma_i$ descends to a character $\sigma_i : G/H \to \bb{C}^*$, we have an induced character $\left.\det X\right|_{\Gamma}$ of $\Gamma$. Then (up to twisting by a power of $\sigma$), $\ll' := \ll^* \oo \det X$ is also a witness representation of $\mc{S}^{\chi'}$, where we put $\chi':=\chi \oo \left.\det X\right|_{\Gamma}$.

Fix a tuple $\underline{k}=(k_1,\dots, k_l) \in \mathbb{N}^l$. We have an equation
\[ \left( \ul{f}^{*\ul{m}}(\partial) \cdot \ul{f}^{\ul{m}} -b_{h,\ul{m}}(\ul{k}) \right) \cdot \ul{f}^{\ul{k}}\cdot h = 0.\]
By Lemma \ref{lem:witpres} and the Nullstellensatz, there exists some power $\ul{p}\in \bb{N}^l$ such that
\[ 0= \ul{f}^{\ul{p}}\cdot \left( \ul{f}^{*\ul{m}}(\partial) \cdot \ul{f}^{\ul{m}} -b_{h,\ul{m}}(\ul{k}) \right) \cdot \, 1 \oo V_{\ul{k} \cdot \ul{\sigma} + \ll}  \,\,\, \subset \D_X \oo_{U\lie}  V_{\ul{k} \cdot \ul{\sigma} + \ll} = P(\ul{k} \cdot \ul{\sigma} + \ll).\]
Put $d=m_1 \deg f_1 + \dots + m_l \deg f_l$. Taking the twisted Fourier transform (see (\ref{eq:fourier}) -- note that the twisting takes care of complex conjugates), we obtain by Lemma \ref{lem:four}
\[ 0= \ul{f}^{*\ul{p}}(\partial) \cdot \left( (-1)^d \cdot \ul{f}^{\ul{m}} \cdot \ul{f}^{*\ul{m}}(\partial) -b_{h,\ul{m}}(\ul{k}) \right) \cdot \, 1 \oo V_{-\ul{k} \cdot \ul{\sigma} + \ll'} \,\,\,  \subset P(-\ul{k} \cdot \ul{\sigma} + \ll').\]
Note that we have $\ul{f}^{\ul{c}-\ul{k}} \cdot h^* \in (\mc{S}^{\chi'})_{-\ul{k} \cdot \ul{\sigma}+\ll'}$. Via the (unique, up to scalar) non-zero map $P(-\ul{k} \cdot \ul{\sigma} + \ll') \to \mc{S}^{\chi'}(*D)$ induced by (\ref{eq:PV}), we get
\[0\!=\! \ul{f}^{*\ul{p}}(\partial) \cdot \left( (-1)^d \cdot \ul{f}^{\ul{m}} \cdot \ul{f}^{*\ul{m}}(\partial) -b_{h,\ul{m}}(\ul{k}) \right) \cdot \ul{f}^{\ul{c}-\ul{k}} \cdot h^*  = \! \left( (-1)^d \cdot b_{h^*,\ul{m}}(\ul{c}-\ul{k}-\ul{m}) -b_{h,\ul{m}}(\ul{k}) \right) \cdot f^{*p}(\partial) \cdot \ul{f}^{\ul{c}-\ul{k}} \cdot h^* .\]
We have $\ul{f}^{*\ul{p}}(\partial) \cdot \ul{f}^{\ul{c}-\ul{k}} \cdot h^* = b_{h^*,\ul{p}}(\ul{c}-\ul{k}-\ul{p}) \cdot \ul{f}^{\ul{c}-\ul{k}-\ul{p}} \cdot h^* \in \mc{S}^{\chi'}(*D)$. Since the latter module is torsion-free, we obtain 
\[\left( (-1)^d \cdot b_{h^*,\ul{m}}(\ul{c}-\ul{k}-\ul{m}) -b_{h,\ul{m}}(\ul{k}) \right) \cdot b_{h^*,\ul{p}}(\ul{c}-\ul{k}-\ul{p}) = 0.\]
As this holds for all $\underline{k}=(k_1,\dots, k_l) \in \mathbb{N}^l$, for sufficiently large tuples $\ul{k}$ we get $b_{h^*,\ul{p}}(\ul{c}-\ul{k}-\ul{p}) \neq 0$ by Proposition \ref{prop:several} (independent of the respective $\ul{p}\in \bb{N}^l$), hence the conclusion.
\end{proof}

The above result generalizes M. Sato's functional equation of $b$-functions \cite[Proposition 4.19]{preh}, which corresponds to the case $h=1$ below (as well as \cite[Lemma 3.4]{graschu}; see also \cite[Lemma 3.1]{gyoja}).

\begin{corollary}\label{cor:bsymm}
Assume that $X\setminus O = D$ is irreducible. Put $d=\deg f$, $n=\dim X$, and let $\ll$ be a witness representation of $\mc{S}^\chi$. Then for non-zero $h \in (\mc{S}^\chi)_\ll$ and $h^* \in (\mc{S}^{\chi^*})_{\ll^*}$, we have
\[b_{h}(s) \,\, = \,\, \pm \, b_{h^*}\left(-s-\frac{n}{d}-1\right).\]
\end{corollary}

\begin{proof}
Put $l=1$, $\ul{m}=1$, in Theorem \ref{thm:bfourier}, and use $c_1= - n/d$ by \cite[Proposition 14]{saki}.
\end{proof}

\begin{remark}\label{rmk:free}
Let us mention some overlap between Corollary \ref{cor:bsymm} and the work \cite{narv}. Assume additionally that $D$ is a linear free divisor (see \cite{graschu}; for example, this holds for the cases (1), (2), (6) in Section \ref{sec:class}). In particular, then the generic stabilizer $H$ is finite, $X\setminus O = D$, and $n=d$. Note that by Proposition \ref{prop:witcri}, the existence of a witness weight $\ll$ for $\chi$ is equivalent to $\mc{S}^\chi(*D)$ being a free $\mc{O}_O$-module (see discussion after Theorem \ref{thm:orbit}). Then the module $\mc{E} = \mc{O}_X \oo V_\ll \,\subset \mc{S}^\chi$ is an integrable logarithmic connection with respect to $D$ (see \cite[Section 5]{narv}). Thus, when the divisor is linear free, Corollary \ref{cor:bsymm} can be deduced from \cite[Theorem 5.1]{narv}.
\end{remark}

\section{Classification of $G$-finite functions on irreducible prehomogeneous vector spaces}\label{sec:class}

We first assume that $X$ is a prehomogeneous vector space. We continue with the notation from the previous section. The hypersurface $D$ is defined by a reduced semi-invariant $f$ of weight $\sigma$, and we denote by $O \subset X\setminus D$ the dense orbit.

We denote by $\Pi(G)$ group of the characters of $G$, viewed as a subgroup of the group of infinitesimal characters of $\lie$. The point of departure of the classification is the following statement.

\begin{prop}\label{prop:normal}
Assume that $D$ is a hypersurface with only normal crossing singularities in codimension one. Let $f=\prod_{i=1}^l f_i$ be its reduced defining equation, with $f_i$ an irreducible semi-invariant with character $\sigma_i$, for $i=1,\dots, l$. Then
\[\Gf_{X}= \bigoplus_{ \substack{\ul{\alpha} \in \, \bb{Q}^l \cap [0,1)^l \\ \ul{\alpha}\cdot \ul{\sigma} \, \in \, \Pi(G)}} \O_X(*D) \cdot \prod_{i=1}^l f_i^{\alpha_i} .\]
\end{prop}

\begin{proof}
Clearly, the right-hand side is always contained in $\Gf_{X}$. By a Theorem of L\^e--Saito (see \cite[Theorem 3.2.9]{dimca}), the assumptions imply that we have $\pi_1(X\!\setminus D) \cong \bb{Z}^l$. By the Riemann--Hilbert correspondence, the simples $S^\chi$ in Corollary \ref{cor:gfinalg} correspond to products of powers via exponentiating characters of $\pi_1(X\!\setminus\!D)$. Since these factor through the finite group $\pi_1^G(X\!\setminus\!D) = \Gamma$, the powers must be rational of the required type. 
\end{proof}

\begin{remark}\label{rem:ratsing}
By a result of Saito \cite[Theorem 0.4]{saito}, the $b$-function gives a criterion for $D$ to have rational singularities, hence a sufficient condition for normality.
\end{remark}

We now proceed case-by-case assuming that $X$ is irreducible prehomogeneous vector space (with $G$ a connected reductive group). Such spaces were classified by Sato and Kimura \cite{saki}. By Theorem \ref{thm:algdecomp}, the classification of $G$-finite functions reduces to finding the solutions of the simple equivariant $\D$-modules $\mc{S}^\chi$, $\chi\in \Lambda(\Gamma)$. We have that $D$ is irreducible \cite[Proposition 4.12]{saki}, and if it is non-empty then $C=\emptyset$ so that $X\setminus D = O$. By Proposition \ref{prop:normal} the only interesting cases are when the irreducible hypersurface $D$ is not normal. Using Remark \ref{rem:ratsing} together with the list of Bernstein--Sato polynomials in \cite{kimu}, as well as the calculations of generic stabilizers \cite{sasada}, \cite[Appendix]{preh}, we have to consider only $6$ cases (reduced, i.e. up to castling transforms):

\begin{enumerate}
\item $(\GL_2,  \, 3\Lambda_1)$ -- binary cubic forms;
\item $(\SL_3 \times \GL_2, \, 2\Lambda_1 \oo \Lambda_1)$ -- pairs of $3\times 3$ symmetric matrices;
\item $(\SL_3 \times \SL_3 \times \GL_2, \, \Lambda_1 \oo \Lambda_1 \oo \Lambda_1)$ -- $2\times 3 \times 3$ tensors;
\item $(\SL_6 \times \GL_2, \, \Lambda_2 \oo \Lambda_1)$ -- pairs of $6\times 6$ skew-symmetric matrices;
\item $(\op{E}_6 \times \GL_2, \, \Lambda_1 \oo \Lambda_1)$ -- pairs of exceptional simple Jordan algebras;
\item $(\SL_5 \times \GL_4, \, \Lambda_2 \oo \Lambda_1)$ -- quadruples of $5\times 5$ skew-symmetric matrices.
\end{enumerate}

Here we used the Bourbaki notation for irreducible highest weight representations. Case (6) is notoriously complicated, and the $b$-function of $f$ (which has degree $40$) was finally obtained only after a series of papers and conjectures \cite{ozeki1,ozeki2,ozeki3}. We leave the respective classification for case (6) for future work.

\subsection{Binary cubics}\label{sec:bin}
As in Section \ref{sec:bin0}, we put $X=\Sym^3 W$ with $G=\GL_2/K$, where $\dim_{\bb{C}} W=2$ and $K=\{\omega I_2 \, \mid \, \omega^3=1 \}$. The action has $4$ orbits:
\begin{itemize}
 \item The zero orbit $O_0 = \{0\}$.
 \item The orbit $O_1 = \{ l^3 : 0\neq l \in W\}$ of cubes of linear forms, whose closure $\ol{O_1}$ is the affine cone over the twisted cubic.
 \item The orbit $O_2 = \{ l_1^2 \cdot l_2 : 0 \neq l_1,l_2 \in W \mbox{ distinct up to scaling}\}$ whose closure $\ol{{O}_2}$ is the hypersurface defined by the vanishing of the cubic discriminant $f$, which is a $\GL_2$-semi-invariant of weight $(6,6)$.
 \item The dense orbit $O = \{ l_1 \cdot l_2 \cdot l_3 : 0\neq l_1,l_2,l_3 \in W \mbox{ distinct up to scaling}\}$.
\end{itemize}

It is known that the hypersurface defined by $f$ is not normal, as its $b$-function is \cite{kimu}
\begin{equation}\label{eq:bbin}
(s+1)^2 \cdot (s+5/6)\cdot (s+7/6),
\end{equation}
and it is also a consequence of the fact that the component group $\Gamma = H/H_0$ of $O=G/H$ is non-abelian, with $H_0=\{1\}$ and $\Gamma = H \cong S_3$, where $S_3$ stands for the symmetric group of order $3$. The group $\Gamma$ has two $1$-dimensional representations $\op{triv}, \op{sign}$ and the standard $2$-dimensional representation $\op{st}$. We have $\mc{S}^{\op{triv}} = \C[X], \, \mc{S}^{\op{sign}} = \mc{S}^{\op{sign}}(*D) = \C[X]_f \cdot \sqrt{f}$ (by (\ref{eq:bbin}) and Lemma \ref{lem:bgen}, see also \cite{bindmod}). 

Although this particular space has been studied in detail in \cite{bindmod}, our purpose is to give an explicit functional interpretation of $\mc{S}^{\op{st}}$ and understand its implications on $\op{mod}_G(\D_X)$, as these methods generalize to the other cases as well. We note the differences between $G$-equivariant and $\GL_2$-equivariant $\D_X$-modules, as in the latter case there are more torsion-free simples, but they can all be obtained easily through tensoring by a suitable rational power of $f$ (see \cite{bindmod}).

We want to find an algebraic function generating a witness representation in $\mc{S}^{\op{st}}$. As it turns out, the natural candidate from Theorem \ref{thm:bin} will do. Consider a root $r(x_0,x_1,x_2,x_3)$, as analytic function in $x_0,x_1,x_2,x_3$ on some domain, of the cubic

\begin{equation}\label{eq:root}
x_3 \cdot y^3 +  x_2 \cdot y^2 + x_1 \cdot y + x_0 =0.
\end{equation}
We note again that $r$ itself is not $G$-finite, but $h(x_0,x_1,x_2,x_3)=x_2 +3x_3 \cdot r(x_0,x_1,x_2,x_3)$ is.

\begin{prop}\label{prop:binary}
Put $h=x_2 + 3x_3 \cdot r$. Then $h \in \mc{S}^{\op{st}}$ is a $G$-finite holonomic function, and $U\lie \cdot h \, \cong V_{(2,1)}$ is a witness representation for $\mc{S}^{\op{st}}$. We have
\[ \partial f \cdot f^{s+1} h = b_h(s) \cdot f^s h, \mbox{ with } \quad b_h(s)=(s+1)^2 \cdot \left(s+\frac{3}{2}\right)^2.\]
\end{prop}

\begin{proof}
By Theorem \ref{thm:bin}, $h \in \mc{S}^{\op{st}}$ is a $G$-finite holonomic function, and $U\lie \cdot h \, \cong V_{(2,1)}$. Proposition \ref{prop:witcri} shows readily that $\ll =(2,1)$ is a witness representation for $\mc{S}^{\op{st}}$, and together with Theorem \ref{thm:bfunmult} it gives the required differential equation for $b_h(s)$, with $\deg b_h(s)=4$.

As $\chi \cong \chi^*$, we have $(\mc{S}^{\chi})_{\ll^*} =1$  (see Theorem \ref{thm:bfourier} and Proposition \ref{prop:witcri}). Note that $h^*:= h/\sqrt{f} \, \in (\mc{S}^{\chi})_{\ll^*}$, and clearly $b_{h^*}(s)=b_h(s-1/2)$. Therefore, by Corollary \ref{cor:bsymm}, we obtain
\[b_h(s)=b_{h^*}(-s-2) = b_h(-s-5/2).\]
By Proposition \ref{prop:binroot}, we know that $(s+1)(s+3/2)$ divides $b_h(s)$, and therefore their quotient is of the form $(s+a)(s+5/2-a)$, for some rational number $a \geq 5/4$ (cf. Corollary \ref{cor:integral}). Put $M = (\D_X \cdot f^{-a}h)/(\D_X \cdot f^{-a+1})$. By Lemma \ref{lem:bgen}, the $\D_X$-module $M$ is non-zero. Clearly, it is an $\SL_2$-equivariant $\D_X$-module supported in $D$. For each $i=0,1,2$, the $G$-orbit $O_i$ is also an $\SL_2$-orbit with its $G$-stabilizer equal to its $\SL_2$-stabilizer (see \cite[Section 3.1]{bindmod}). Thus, any composition factor of $M$ is a $G$-equivariant $\D_X$-module, which forces $2a \in \bb{Z}$. From the character formula of $H_{\ol{O}_1}^2(\mc{O}_X)$ (see \cite[Section 3.2]{bindmod}), we see that $\mc{S}^{\chi}(*D)\cong(\D_X \cdot f^{-a}h)$ is generated by both $f^{-1}h$ and $f^{-3/2}h$ as a $\D_X$-module. By Lemma \ref{lem:bgen} again, we conclude that $a=3/2$.
\end{proof}

The proof above illustrates the circle of ideas that we have developed, but since we have explicit generators of $\Ann(h)$ by Theorem \ref{thm:bin}, they can be used to calculate $b_h(s)$ through a computer algebra system (this direct approach does not terminate for rest of the series (2)--(6)).  Conversely, given $b_h(s)$ we can obtain various results from \cite{bindmod} by different means. For example, Lemma \ref{lem:bgen} then readily implies
\[\D_X \cdot h/\sqrt{f}\, \cong \, \D_X \cdot h \cong \mc{S}^{\op{st}} \,\, \subsetneq \,\, \mc{S}^{\op{st}}(*D)\cong\D_X \cdot f^{-1} h \, \cong \, \D_X \cdot f^{-3/2} h.\]

\subsection{The series (2)--(5)} \label{sec:hard}

As in \cite{kimuro}, we treat the cases (2),(3),(4),(5) as a series depending on the parameter $l=1,2,4,8$, respectively. We have $\dim X = 6l+6$, and we quotient out by the kernel of the action map so that $G$ acts faithfully. The component group is $\Gamma \cong S_4$ for $l=1$ and $\Gamma \cong S_3$ for $l=2,4,8$ by \cite{sasada} (see also \cite[Appendix]{preh}). As in the binary cubics case, for $l=2,4,8$ the interesting torsion-free equivariant simple $\D$-module is $\mc{S}^{\op{st}}$ corresponding to the standard representation of $S_3$. 

On the other hand, the case $l=1$ is slightly degenerate, as can be also seen from the corresponding holonomy diagram \cite{kimuro}, and is a phenomenon that appears in the subexceptional series in \cite{locsub} as well. For uniformity, we still denote the simple $\D$-module corresponding to the $2$-dimensional irreducible representation of $S_4$ by $\mc{S}^{\op{st}}$, since it factors through the group morphism $S_4 \to S_3$. We denote by $\mc{S}^{(3,1)}$ the simple $\D$-module corresponding to the standard representation of $S_4$. Note that the simple $\D$-module $\mc{S}^{(2,1,1)}$ corresponding to the other irreducible $3$-dimensional irreducible representation of $S_4$ is related to $\mc{S}^{(3,1)}$ through multiplication by $\sqrt{f}$ in $\Gf_{X}$, since $\mc{S}^{(2,1,1)}(*D) \cdot \sqrt{f}  \cong \mc{S}^{(3,1)}(*D)$ as $\D_O$-modules.

The representation $X$ has the following uniform description in terms of Hurwitz algebras $\bb{R}, \bb{C}, \bb{H}, \bb{O}$ \cite{kimuro}. For $l=1,2,4,8$, let $A$ denote $\bb{R} \oo_{\bb{R}} \bb{C},\, \bb{C} \oo_{\bb{R}} \bb{C}, \, \bb{H} \oo_{\bb{R}} \bb{C}, \, \bb{O} \oo_{\bb{R}} \bb{C}$, respectively. We identify $X$ with the space of pairs of $3\times 3$ hermitian matrices $(X_1,X_2)$ with the natural action $\rho : \SL_3(A) \times \GL_2 \to \GL(X)$, and we put $G= (\SL_3(A) \times \GL_2) / \ker \rho$. The irreducible semi-invariant $f$ of degree $12$ is the discriminant of the binary cubic
\[\det(uX_1 + v X_2) = d_1 \cdot u^3 +  d_2 \cdot u^2 v + d_3 \cdot uv^2 + d_4 \cdot v^3.\]

Based on standard methods from representation theory, we obtain the following result.

\begin{lemma}\label{lem:poly}
The polynomials $d_1, d_2, d_3, d_4$ are algebraically independent and $\C[X]^{\SL_3(A)}=\C[d_1,d_2,d_3,d_4]$. 
\end{lemma}

By \cite{kimuro}, the Bernstein--Sato polynomial of $f$ is
\[(s+1)^2\left(s+\frac{5}{6}\right)\left(s+\frac{7}{6}\right)\left(s+\frac{l+1}{2}\right)^2\left(s+\frac{l+2}{4}\right)^2\left(s+\frac{l+4}{4}\right)^2\left(s+\frac{3l+2}{6}\right)\left(s+\frac{3l+4}{6}\right).\]

We denote by $h(x_0,x_1,x_2,x_3)=x_2 +3x_3 \cdot r(x_0,x_1,x_2,x_3)$ the function from Proposition \ref{prop:binary}, where $r$ is the root of the cubic (\ref{eq:root}). We obtain the following uniform result in terms of the parameter $l=1,2,4,8$.

\begin{theorem}\label{thm:paramb}
Put $h_0=h(d_1,d_2,d_3,d_4)$. Then $h_0 \in \mc{S}^{\op{st}}$ is a $G$-finite holonomic function with $U\lie \cdot h_0 \, \cong \op{triv}\oo  V_{(2,1)}$ a witness representation for $\mc{S}^{\op{st}}$. We have
\[b_{h_0}(s)\!=\!(s+1)^2\left(\!s+\frac{3}{2}\right)^{\!2}\!\!\left(\!s+\frac{l+1}{2}\right)^{\!2}\!\!\left(\!s+\frac{l+2}{2}\right)^{\!2}\!\!\left(\!s+\frac{3l+8}{12}\right)\!\left(\!s+\frac{3l+10}{12}\right)\!\left(\!s+\frac{3l+14}{12}\right)\!\left(\!s+\frac{3l+16}{12}\right)\!\!.\]
\end{theorem}

We explain how we obtained the Bernstein--Sato polynomial, as all the other claims can be checked easily. The computation of $b_{h_0}(s)$ is performed by the method of reducing invariant differential operators through passing to an affine quotient. Here we only sketch this technique, and give some further details for the case in Theorem \ref{thm:paramb1}, as it becomes more intricate there.

By Theorem \ref{thm:bfunmult}, we have an equation of the form
\begin{equation}\label{eq:start}
\partial f \cdot f^{s+1} h_0 = b_{h_0}(s) \cdot f^s h_0.
\end{equation}
Consider the quotient map $\,\,p: X \longrightarrow X/\!/\SL_3(A)$. Note that by Lemma \ref{lem:poly} the space $X/\!/\SL_3(A)$ can be identified with the space of binary cubic forms, and we have $h_0=h \circ p$ and $f=d\circ p$, where $d$ is the discriminant of the binary cubic from Section \ref{sec:bin}. There is an induced algebra map
\[ P: \quad \D_X^{\SL_3(A)} \longrightarrow \D_{X/\!/\SL_3(A)} .\]
The equation (\ref{eq:start}) on $X$ descends to an equation on  $X/\!/\SL_3(A)$ \footnote{This explains the relation $b_{h}(s) \mid b_{h_0}(s)$. The same reasoning explains why the $b$-function (\ref{eq:bbin}) divides that of $f$.}
\begin{equation}\label{eq:red}
P(\partial f) \cdot d^{s+1} h = b_{h_0}(s) \cdot d^{s} h.
\end{equation}
First, we explain how to obtain $Q=P(\partial f)$. As the degree of $f$ is $12$, so is the degree of the differential operator $Q$. In order to write down $Q$ in the Weyl algebra $\D_{X/\!/\SL_3(A)}$ explicitly, in principle we need to evaluate $Q \cdot d_1^{s_1}d_2^{s_2}d_3^{s_3}d_4^{s_4}$ and express the result in $\C[d_1,d_2,d_3,d_4]$, for each $\ul{s}=(s_1,s_2,s_3,s_4) \in \bb{N}^4$ with $s_1+s_2+s_3+s_4 \leq 12$. The powers $\ul{s}$ are too large for calculations to terminate on a computer, but we can make the following simplification. As $f \in \C[d_1,d_2,d_3,d_4]$, we also have $\partial f \in \C[\partial d_1, \partial d_2, \partial d_3, \partial d_4]$ for $\SL_3(A)$-invariants $\partial d_1, \partial d_2, \partial d_3, \partial d_4$ dual to $d_1,d_2,d_3,d_4$. Since $P$ is an algebra map, in order to get $Q$, it is enough to compute $P(\partial d_i)$ for each $i=1,2,3,4$, which now requires to evaluate $P(\partial d_i) \cdot d_1^{s_1}d_2^{s_2}d_3^{s_3}d_4^{s_4}$ only for $\ul{s}$ with $s_1+s_2+s_3+s_4 \leq 3$. These evaluations do terminate, and by writing the obtained results in terms of $d_1,d_2,d_3,d_4$, we can inductively construct each $P(\partial d_i)$ as the entries of $\ul{s}$ increase. 

While the calculations involved in obtaining $P(\partial d_i)$ are in principle computable by hand, we developed a program in \defi{Macaulay2} to handle the whole process automatically in a more general setting.

Once $Q$ is determined, we rewrite (\ref{eq:red}) as $Q \cdot d - b_{h_0}(s) \,\, \in \, \Ann(d^s h)$. By Theorem \ref{thm:bin}, the ideal $I=(g_{11}-6s-2, g_{12}^2, g_{21}, g_{22}-6s-1) \, \subset \D_{\bb{C}^4}[s]$ is contained in $\Ann(d^s h)$. Computing the normal form of $Q \cdot d$ with respect to a Gr\"obner basis of $I$, we obtain $b_{h_0}(s)$ (up to a constant).

\medskip

Now we consider the remaining simple $\D$-modules $\mc{S}^{(2,1,1)}$ and $\mc{S}^{(3,1)}$ in the case $l=1$.

\begin{theorem}\label{thm:paramb1}
The representation $\ll^1 = \Lambda_1 \oo (1,1)$
 (resp. $\ll^2 = \Lambda_1 \oo (-2,-2)$) is a witness representation for $\mc{S}^{(3,1)}$ (resp. $\mc{S}^{(2,1,1)}$). For a non-zero $h_1\in (\mc{S}^{(3,1)})_{\ll^1}$ (resp. $h_2\in (\mc{S}^{(2,1,1)})_{\ll^2}$),  we have
\[ b_{h_1}(s)=(s+1)^4\left(s+\frac{3}{2}\right)^4 \left(s+\frac{5}{6}\right)^2\left(s+\frac{7}{6}\right)^2\!, \text{ and } \quad b_{h_2}(s)=b_{h_1}\left(s-\frac{1}{2}\right).\]
\end{theorem}

In the rest of this subsection we outline the proof of Theorem \ref{thm:paramb1} above. The fact that $\ll^1$ (resp. $\ll^2$) is a witness representation follows from Proposition \ref{prop:witcri}, by inspection. We can assume that $h_1, h_2$ are highest weight vectors. Since $h_1 = \sqrt{f} \cdot h_2$ (up to a constant), it is enough to compute $b_{h_1}(s)$. We can further choose $h_1$ such that its orbit under the Galois group has $4$ elements only. Therefore, $y=h_1$ must satisfy an algebraic equation of the form (the constants are chosen for convenience)
\begin{equation}\label{eq:alg}
y^4-2\cdot v_4 \cdot y^2+8 \cdot v_6 \cdot y-v_8 =0.
\end{equation}
Each $v_i$ must be a degree $i$ polynomial of highest weight $(i/2) \cdot \ll^1$. By a plethysm calculation, we see that both the $2 \ll^1$ and $3 \ll^1$ isotypical components of $\C[X]$ have dimension one, while the $4 \ll^1$ component dimension two. Therefore, finding the equation (\ref{eq:alg}) amounts to just finding 4 constants.

The fact that $h_1$ is $G$-finite implies, in particular, that the discriminant of (\ref{eq:alg}) is divisible by $f$. It turns out, this is enough to find the 4 constants (up to a scaling factor), which is done by an explicit calculation.

 We think of the space $X$ as pairs of $3\times 3$ symmetric matrices $\{(x_{ij}), (y_{ij})\}$. Write
 \[Z=\begin{pmatrix}
 x_{11} & x_{12} & x_{13} &x_{22} &x_{23} &x_{33} \\
 y_{11} & y_{12} & y_{13} &y_{22} &y_{23} &y_{33}
 \end{pmatrix}.\]
We number the colums of $Z$ from $0$ to $5$. Let $d_{ij}$ denote the $2\times 2$ minor of $Z$ formed by the columns $i$ and $j$, with $0 \leq i \leq j \leq 5$. The calculation yields the following:
\[\begin{array}{ll}
\!\!\!v_4 &\!\!\!\!=d_{04}^2-d_{03}d_{05}+2d_{02}d_{23}+2d_{04}d_{12}+3d_{12}^2-4d_{02}d_{14}+2d_{01}d_{15}, \\
\!\!\!v_6 &\!\!\!\!=d_{04}d_{12}^2\!+\!d_{12}^3\!+\!d_{02}d_{05}d_{13}\!-\!d_{01}d_{05}d_{14}\!-\!2d_{02}d_{12}d_{14}\!-\!d_{02}d_{03}d_{15}\!+\!d_{01}d_{04}d_{15}\!+\!d_{01}d_{12}d_{15}\!+\!d_{02}d_{23}d_{12}\!+\!\!d_{02}^2d_{34},\\
\!\!\!v_8 &\!\!\!\!=(-1/3)\cdot(4y_{11}^2d_2^2-12y_{11}^2d_1d_3-4x_{11}y_{11}d_2d_3+4x_{11}^2d_3^2+36x_{11}y_{11}d_1d_4-12x_{11}^2d_4-v_4^2).
\end{array}\]
By Theorem \ref{thm:bfunmult}, we have an equation of the form
\begin{equation}\label{eq:desthis}
\partial f \cdot f^{s+1} h_1 = b_{h_1}(s) \cdot f^s h_1.
\end{equation}
As in the previous case, we reduce (\ref{eq:desthis}) to a smaller space, inspired by invariant theory. Consider the subgroup $H\subset \GL_3 \subset \GL_3 \times \GL_2$ consisting of matrices of the form \footnote{We have tried using smaller subgroups and also $\SL_2$ (on the second factor), but the computations did not terminate then.} 
\[H \, = \, \left\{\begin{pmatrix}
1 & a \\
0 & A
\end{pmatrix}: \,\, a \in \C^2 \, \mbox{ and } \det(A) \in \{-1,1\} \right\}.\]

The following can be shown with the help of the computer, for example.

\begin{lemma}\label{lem:subpol}
The elements $x_{11},y_{11},d_1,d_2,d_3,d_4,v_4 \in \C[X]^H$ are algebraically independent.
\end{lemma}

We denote by $R=\C[x_{11},y_{11},d_1,d_2,d_3,d_4,v_4] \subseteq \C[X]^H$ \footnote{We believe that this is actually an equality, and while this would simplify the argument a bit, it is not essential.}. We have $f,v_8 \in R$, and we also get that $v_6^2 \in R$ by an explicit verification (it is clear that $v_6^2 \in \C[X]^H$). We write $f=d(d_1,d_2,d_3,d_4)$, where $d$ is the discriminant of the binary cubic, and we have the dual version $\partial f = d(\partial d_1, \partial d_2, \partial d_3, \partial d_4)$.

Through the computer program that we developed and mentioned already in the previous case, we can compute each $\partial d_i \cdot  x_{11}^{s_1}y_{11}^{s_2}d_1^{s_3}d_2^{s_4}d_3^{s_5}d_4^{s_6} v_4^{s_7}$ for $\ul{s} \in \bb{N}^7$ with $\sum_{i=1}^7 s_i \leq 3$. Moreover, we also verified that each result lies in $R$ (clearly, it lies in $\C[X]^H$). Therefore, we can reduce (\ref{eq:desthis}) to an equation on $\op{Spec}(R) = \C^7$, where we denote by $Q$ the differential operator induced by $\partial f$. 

Now we consider a series expansion of $y$ as given in \cite{sturm} \footnote{We have implemented also various Gr\"obner methods, but none of these computations terminated.}. By a careful argument using highest weights, it is not difficult to show that in order to obtain $b_{h_1}(s)$ from (\ref{eq:desthis}), it is enough to consider a truncation of the series expansion of $h_1$, apply $\partial f \cdot f^{s+1}$ to it, and then identify the coefficient of the first monomial in the result. More precisely, it suffices to find the coefficient of the monomial $f^s \sqrt{w}/u$ after applying $\partial f \cdot f^{s+1}$ to the following truncation of $h_1$ (here $u=-2v_4,z=-v_8,w=64v_{6}^2$)
\[ T= \sqrt{w}\cdot \left[u^{-1} - w/u^4 + (2z)/u^3 - (10wz)/u^6 + (6z^2)/u^5 - (70wz^2)/u^8 - (330wz(w^2 + 7z^3))/u^{12} + \right. \]
\[+(3w^2 + 20z^3)/u^7 \!- 12(w^3 + 35wz^3)/u^{10}\! + 14(4w^2z + 5z^4)/u^9\!+ 126(5w^2z^2 + 2z^5)/u^{11}\! - 1716(3w^3z^2 + 7wz^5)/u^{14}\!+\] 
\[+11(5w^4 +504w^2z^3 + 84z^6))/u^{13} - 273(w^5 + 220w^3z^3 + 220wz^6)/u^{16} + 286(7w^4z + 147w^2z^4 +
12z^7)/u^{15} - \]
\[ - 1768(7w^5z + 330w^3z^4 + 165wz^7)/u^{18} + 286(140w^4z^2 + 1008w^2z^5 + 45z^8)/u^{17} - 25194(12w^5z^2 + 198w^3z^5 + \]
\[+ 55wz^8)/u^{20} + 4522z^2(14157w^2z^6 + 156z^9)/u^{23} + 646z(10725w^4z^3 + 17160w^2z^6 + 286z^9)/u^{21} +\] 
\[\left.+ 68(21w^6 + 8580w^4z^3 + 27027w^2z^6 + 715z^9)/u^{19} - 1292(30030w^3z^6 + 5005wz^9)/u^{22} \right].\]
Write $T=\sqrt{w} \cdot T'$. Viewing $T$ on $\op{Spec}(R)=\bb{C}^7$, we see that $T'$ is a Laurent monomial. In principle, we can apply $Q$ to $T$ (even to the algebraic function $y$ on $\op{Spec}(R)$) using the  \defi{HolonomicFunctions} package, but this does not terminate, as the factor $\sqrt{w}$ creates a computational bottleneck.

Therefore, instead of reducing $\partial d_i$ to $R$, we perform the following trick.

\begin{lemma}\label{lem:trick}
We have a well-defined differential operator $\frac{1}{v_6} \cdot \partial d_i \cdot v_6 \, : \, R \longrightarrow R$, for $i=1,2,3,4$.
\end{lemma}

We use again our computer program to prove this and compute each $\frac{1}{v_6} \cdot \partial d_i \cdot v_6$ on $\op{Spec}(R)$, which we denote by $Q_i$. Then (\ref{eq:desthis}) reduces to the following on $\C^7$ via truncating
\begin{equation}\label{eq:lasteq}
d(Q_1,Q_2,Q_3,Q_4) \cdot d(d_1,d_2,d_3,d_4)^{s+1}  \cdot T' \,\,\,= \,\,\, \frac{b_{h_1}(s)}{u} \cdot d(d_1,d_2,d_3,d_4)^{s} + \mbox{ \,other irrelevant terms}.
\end{equation}

Even with the knowledge of $Q_i$, evaluating $d(Q_1,Q_2,Q_3,Q_4)$ in $\D_{\C^7}$ is computationally expensive. Since it is enough to evaluate the right-hand side of (\ref{eq:lasteq}) at a point in $\C^7$ (with as many zero entries as possible) such that $u \neq 0 \neq d(d_1,d_2,d_3,d_4)$ and $w=0,z=0$, we further wrote a program in \defi{Macaulay2} that takes this into account while working in the Weyl algebra on the left hand side of (\ref{eq:lasteq}). Once we performed all of these reductions, we finally found $b_{h_1}(s)$ using the \defi{HolonomicFunctions} package.

\begin{remark} Using the method of reduction by invariant differential operators as above, we can equally compute the $b$-functions of the semi-invariants $f$ as well, thus offering an alternative approach to the microdifferential method for many cases in \cite{kimu}.
\end{remark}

\subsection{$G$-finite functions under castling transformations}

We now discuss the behavior of algebraic functions under castling transforms. At first, $G$ can be any linear algebraic group.

Let $(\pi,V)$ and $(\rho,W)$ be two finite dimensional representations of $G$ with $\dim V=n$. We denote by $\Lambda_1$ the standard representation of general linear group. Take two numbers $n_1,n_2\in\bb{N}$ with $n_1+n_2=n$, and put $(\pi',V')$ to be the representation of $G$ corresponding to $V':=V^* \oo (\det V)^\frac{1}{n_2}$. We can form two representations (see \cite[Section 2.3]{kac})\footnote{Technically, $(\det V)^\frac{1}{n_2}$ is only a $\lie$-module, but we will continue calling $V'$ a $G$-module. The twist by $(\det V)^\frac{1}{n_2}$ is convenient to give $G$-equivariant correspondences, otherwise \cite[Proposition 2.1]{kac} does not hold as stated. For example, take $n_1=n_2=1, n=2,$ with $G=\bb{C}^*$ acting on $V=\bb{C}^2$ and $W=\bb{C}$ by scalar multiplication.}
$$X_1=(G\times \GL_{n_1},\,(\pi\otimes \Lambda_1)\oplus (\rho\otimes 1), \, V^{\oplus n_1}\oplus W),$$
$$X_2=(G\times \GL_{n_2},\,(\pi'\otimes \Lambda_1)\oplus (\rho\otimes 1), \, (V')^{\oplus n_2}\oplus W).$$
Following \cite{saki}, the representations $X_1,X_2$ are called \defi{castling transforms} of each other. In the representation theory of algebras the corresponding functors are called reflection functors. As in \cite[Proposition 2.1]{kac}, we have a $G$-equivariant isomorphism $\psi$ of graded algebras (graded by their $\GL_{n_i}$-weights)
\begin{equation}\label{eq:isom}
\quad \begin{array}{ll}
\bb{C}[X_1]^{\op{SL}_{n_1}} \,\, \xrightarrow{\,\cong\,} \,\, \bb{C}[X_2]^{\op{SL}_{n_2}}\, & \text{ when } n_1,n_2>0, \\[0.2cm] 
\bb{C}[X_1]^{\op{SL}_{n_1}} \,\, \xrightarrow{\,\cong\,} \,\, \bb{C}[X_2]^{\op{SL}_{n_2}}\oo\bb{C}[\op{det}_{n_1}], & \text{ when } n_2=0.
\end{array}
\end{equation}
For simplicity, we will assume $n_1, n_2>0$ throughout, but the formulas extend also to the degenerate cases. 

The paper \cite{bfunquiv} gives relations between the $b$-functions of semi-invariants of prehomogeneous spaces related under castling transforms as above (see also \cite{kimu,saoc}, and \cite[Theorem 7.51]{preh}). Here we generalize the results to multiplicity-free algebraic functions.

First, we extend $\psi$ to an isomorphism for algebraic functions. Fix an $\SL_{n_1}$-stable hypersurface $D\subset X_1$ defined by a (reduced) polynomial, which must be an $\SL_{n_1}$-invariant. We denote the $\SL_{n_2}$-stable hypersurface defined by the latter with $\psi(D) \subset X_2$. The following is a consequence of Proposition \ref{prop:invalg}.

\begin{lemma}\label{lem:casalg}
The map $\psi$ extends to a $G$-equivariant isomorphism of algebras 
\[\psi: \, (\alg_{X_1}(D))^{\mathfrak{sl}_{n_1}} \xrightarrow{\, \cong \, } (\alg_{X_2}(\psi(D)))^{\mathfrak{sl}_{n_2}}.\]
\end{lemma}

We can now formulate the main result of this subsection. We assume again that $G$ is a connected reductive group, and denote by $G_i = G\times \GL_{n_i}$ the group acting on $X_i$. Let $f_1,\dots,f_l \in \bb{C}[X_1]^{\SL_{n_1}}$ (resp. $f_1',\dots,f_l' \in \bb{C}[X_2]^{\SL_{n_2}}$) define the irreducible components of $D$ (resp. $\psi(D)$). Each $f_i$ (resp. $f_i'$) is semi-invariant with respect to $\GL_{n_1}$ (resp. $\GL_{n_2}$) of weight $d_i \in \bb{N}$, say.

\begin{theorem}\label{thm:castle}
Assume that $h$ is a $G_1$-multiplicity-free algebraic function on some domain of $X_1$ of weight $\ll^1=\ll \oo \det^d$, for some $\ll \in \Lambda(G)$ and $d \in \bb{Z}$, and let $D=\Sing \D_{X_1} h$. Then $\psi(h) \in \alg_{X_2}(\psi(D))$, and for any tuple $\underline{m} \in \bb{N}^l$ we have
$$b_{h,\,\ul{m}}(\ul{s}) \, \cdot \, \prod_{i=1}^{n_2} \prod_{j=0}^{\, \ul{d}\cdot \ul{m}-1} (\ul{d}\cdot \ul{s}+d+i+j) \, = \, b_{\psi(h), \, \ul{m}}(\ul{s})  \, \cdot \, \prod_{i=1}^{n_1} \prod_{j=0}^{\, \ul{d}\cdot \ul{m}-1} (\ul{d}\cdot \ul{s}+d+i+j).$$
\end{theorem}

\begin{proof}
It follows from Lemma \ref{lem:casalg} readily that $\psi(h) \in \alg_{X_2}(\psi(D))$ with weight $\ll \oo \det^d$. We briefly recall some facts from the proof of \cite[Theorem 2.1]{bfunquiv}.

Let $A$ (resp. $A'$) be the subring of $\C[X_1]$ (resp. $\C[X_2]$) generated by the maximal minors of the space of $n \times n_1$  (resp. $n\times n_2$) matrices. Let $A_k$ (resp. $A'_k$) denote the respective homogeneous parts of degree $n_1k$ (resp. $n_2k$). Similarly, we define the ring of (constant) differential operators $B$ (resp. $B'$) generated by the maximal minors in the partial variables, and $B_k$ (resp. $B'_k$) denote the homogeneous part of degree $n_1k$ (resp. $n_2k$). The map $\psi$ restricts to a $\GL_n$-equivariant isomorphism of graded algebras $\psi: A\to A'$. Dually, we also have a $\GL_n$-equivariant isomorphism of graded algebras $\psi': B \to B'$.

Fix $k \in \bb{N}$. For any $p\in \bb{N}, p\geq k$,  we have a $\GL_n$-equivariant map $\phi_{p}:B_k\otimes A_p \to A_{p-k}$ (resp. $\phi'_{p}:D'_k\otimes A'_p\to A'_{p-k}$) given by applying differential operators. So $\phi_{p}$ and $\tau_p : = \psi^{-1}\circ \phi' _{p}\circ (\psi'\otimes \psi)$ are two $\GL_n$-module morphisms $B_k\otimes A_p \to A_{p-k}$, and by Schur's lemma they agree up to a constant, e.g. calculated in \cite[Theorem 7.51]{preh}, \cite[Theorem 2.1]{bfunquiv}. Denote by $\phi : B_k \otimes A \to A$ (resp $\tau: B_k \otimes A \to A$) the sum of all maps $\phi_{p}$ (resp. $\tau_{p})$ over $p$. Hence, for any $Q\in B_k$ and $P \in A_p$ we have
\[\prod_{i=0}^{k-1} \prod_{j=0}^{n_2-1} (p-i+j)\, \cdot \phi(Q\oo P) =\prod_{i=0}^{k-1} \prod_{j=0}^{n_1-1} (p-i+j) \, \cdot \tau(Q \oo P).\]
We now extend the domain of $\phi$ (resp. $\tau$) gradually to semi-invariant algebraic functions on $X_1$ (by abuse of notation, we will use the same letters). First, as $\C[X_1]^{\SL_{n_1}}=A \oo \C[W]$ by the First Fundamental Theorem for $\op{SL}$ (cf. \cite{popvin}), we note that the maps $\phi$ and $\tau$ extend naturally to $\phi, \tau: B_k \oo \C[X_1]^{\SL_{n_1}}  \to \C[X_1]^{\SL_{n_1}}$ with $B$ acting trivially on $\C[W]$.

The above equation implies that for any $P_1, \dots ,P_m$ with $P_i \in A_{p_i}\oo \C[W]$ (with $p_i \in \bb{N}$), we have
\begin{equation}\label{eq:variable}
\prod_{i=0}^{k-1} \prod_{j=0}^{n_2-1} (\ul{p} \cdot \ul{s} -i+j)\,\, \cdot \phi(Q\oo (P_1^{s_1}\cdots P_m^{s_m})) =\prod_{i=0}^{k-1} \prod_{j=0}^{n_1-1} (\ul{p} \cdot \ul{s}-i+j) \,\, \cdot \tau(Q \oo (P_1^{s_1}\cdots P_m^{s_m})),
\end{equation}
for all $(s_1,\dots,s_m) \in \bb{N}^m$. But then the same equation must hold by letting $(s_1,\dots,s_m)$ to be a tuple of variables (extending the domains of maps appropriately), and so must hold for tuples of rational numbers.

Now let $y$ be a $\GL_{n_1}$-semi-invariant algebraic function of weight $\det^d$, and write $y^t + a_1 y^{t-1} \dots + a_{t-1} y+a_t$ for the minimal monic polynomial of $y$. As seen in Proposition \ref{prop:invalg}, the coefficient $a_i \in \C(X_1)$ is also $\GL_{n_1}$-semi-invariant, of weight $\det^{i \cdot d}$, for $i=1, \dots, t$. In particular, each $a_i$ can be written as the quotient of two elements in $A\oo \C[W]$ (see \cite[Theorem 3.3]{popvin}).

Next, let $\mc{L} \subset \bb{Z}^{t+1}$ be the sublattice spanned by the vectors $e_{i-1}-2e_i+e_{i+1}$, where $i=1,\dots,t-1$. We observe that there is a vector $\ul{u} \in  \bb{Q}^{t+1}$ such $y$ admits Puiseux series expansions in $1=a_0, \, a_1,\dots, \, a_t$ with the property that the exponents all lie in $\ul{u} + \mc{L}$ (see \cite[Lemma 1]{gkz} and \cite{sturm}). More specifically, one can take the coarsest triangulation in \cite[Theorem 3.2]{sturm} to avoid potential zeroes in denominators. For an exponent $\ul{v} \in \, \ul{u} + \mc{L}$, write $\ul{a}^{\ul{v}}$ for the corresponding Puiseux term, and put $\ul{\delta} = (0, \, d, \, 2\cdot d ,\dots, \, t\cdot d)$. Since $\ul{v} \cdot \ul{\delta} = \ul{u} \cdot \ul{\delta}$,  from (\ref{eq:variable}) we deduce that for any $P_1, \dots ,P_m$ with $P_i \in A_{p_i}\oo \C[W]$ we have
\[
\prod_{i=0}^{k-1} \prod_{j=0}^{n_2-1} (\ul{p} \cdot \ul{s}+\ul{u} \cdot \ul{\delta}-i+j)\,\, \cdot \phi(Q\oo (P_1^{s_1}\cdots P_m^{s_m}\cdot \ul{a}^{\ul{v}})) =\prod_{i=0}^{k-1} \prod_{j=0}^{n_1-1} (\ul{p} \cdot \ul{s}+\ul{u} \cdot \ul{\delta}-i+j) \,\, \cdot \tau(Q \oo (P_1^{s_1}\cdots P_m^{s_m} \cdot \ul{a}^{\ul{v}})).
\]
Thus, summing over all terms of $y$ in its Puiseux expansion, we obtain
\[
\prod_{i=0}^{k-1} \prod_{j=0}^{n_2-1} (\ul{p} \cdot \ul{s}+d-i+j)\,\, \cdot \phi(Q\oo (P_1^{s_1}\cdots P_m^{s_m}\cdot y)) =\prod_{i=0}^{k-1} \prod_{j=0}^{n_1-1} (\ul{p} \cdot \ul{s}+d-i+j) \,\, \cdot \tau(Q \oo (P_1^{s_1}\cdots P_m^{s_m} \cdot y)).
\]
Now putting $k = \ul{d} \cdot \ul{m}, \, Q = \ul{f}^{* \ul{m}}(\partial), \, m=l, \, P_i =f_i,  \, s_i \to s_i+m_i,$ and $y = h$ yields the result according to Proposition \ref{prop:several}.
\end{proof}

\begin{remark}\label{rem:notmult}
The proof above shows that $\psi(h)$ also satisfies the equation as in Proposition \ref{prop:several}, yet, in principle, it might not be multiplicity-free, since $\deg f_i'$ may be larger than $\deg f_i$. Nevertheless, the statement is entirely symmetric, since the existence of such an equation for $h$ is the only requirement for Theorem \ref{thm:castle} to hold, and multiplicity-freeness was only used in order to guarantee this.
\end{remark}

Assume now that $X_1$ is prehomogeneous under the action of $G_1$. By \cite[Propositions 7, 9]{saki}, the space $X_2$ is also prehomogeneous under the action of $G_2$, and their generic stabilizers $\Gamma$ agree.

Write $X_i \setminus O_i = D_i \, \cup \, C_i$, with $\codim_{X_i} C_i \geq 2$. Note that $\psi(D_1) = D_2$. Let $\sigma$ be the weight of the semi-invariant defining $D_1$ (and $D_2$). We denote the correspondence between the simple equivariant torsion-free $\D$-modules by $\mc{S}_1^\chi \mapsto \mc{S}_2^\chi$, for all $\chi \in \Lambda(\Gamma)$.

\begin{lemma}\label{lem:castlewit}
Assume that $\ll^1 = \ll \oo \det^d$ is a witness representation of $G_1$ for $\mc{S}^\chi_1$, for some $\ll \in \Lambda(G)$ and $d\in \bb{Z}$. Then for some $p \in \bb{N}$,\,  $\ll^2+p \cdot \sigma$ is a witness representation of $G_2$ for $\mc{S}_2^{\chi}$,  with $\ll^2 = \ll \oo \det^d$.
\end{lemma}

\begin{proof}
First, note that the isomorphism in Lemma \ref{lem:casalg} commutes with monodromy. Therefore, by Corollary \ref{cor:gfinalg}, for any $\chi' \in \Lambda(\Gamma)$ we have $G$-isomorphisms 
\[\left(\mc{S}_1^{\chi'}(*D_1)\right)^{\mathfrak{sl}_{n_1}} \,  \cong \, \left(\mc{S}_2^{\chi'}(*D_2)\right)^{\mathfrak{sl}_{n_2}}.\]
By Lemma \ref{lem:witness}, twisting by a suitable power of $\sigma$ so that $\ll^2+p \cdot \sigma \in \mc{S}_2^\chi$, we obtain the desired result.
\end{proof}

\begin{example}\label{ex:castle}
We take case (2) from our series, so $G_1 = \SL_3 \oo \GL_2$ acting on $X_1 = 2 \Lambda_1 \oo \Lambda_1$, and the irreducible semi-invariant has weight $\sigma = \op{triv} \oo \det^6$. By Theorem \ref{thm:paramb1} that $\ll^1 = \Lambda_1 \oo \det^2$ is a witness representation for $\mc{S}_1^{(3,1)}$. Then we have $G_2 =  \SL_3 \oo \GL_4$ acting on $X_2 = 2 \Lambda_1 \oo \Lambda_1$, and we let $\ll^2 =\Lambda_1 \oo \det^2$. By Theorems \ref{thm:paramb1}, \ref{thm:castle} and Lemma \ref{lem:castlewit}, for non-zero $h' \in (\mc{S}_2^{(3,1)}(*D_2))_{\ll^2}$ we get
\[b_{h'}(s) = (s+1)^4\left(s+\frac{3}{2}\right)^4 \left(s+\frac{5}{6}\right)^2\left(s+\frac{7}{6}\right)^2 \cdot \, \prod_{i=5}^{6} \, \prod_{j=0}^5 \left(s+\frac{i+j}{6}\right).\]
Since $b_{h'}(s)$ has only negative roots, we see by Lemma \ref{lem:bgen} and Proposition \ref{prop:witcri} that in fact $h' \in \mc{S}_2^{(3,1)}$, so that $\ll^2$ is also a witness representation for $\mc{S}_2^{(3,1)}$.
\end{example}

\section*{Acknowledgments}

I am grateful to Bernd Sturmfels for inspirational conversations about Weyl closure and holonomic functions. I thank Uli Walther for providing me valuable suggestions and comments. I also thank Nero Budur for pointing out some of the relevant literature concerning $V$-filtrations.

\bibliographystyle{alpha}
\bibliography{biblo}

\end{document}